\documentclass{article}

\usepackage{quoting}
\usepackage{setspace}
\usepackage[utf8]{inputenc} 
\usepackage{amsmath}
\usepackage{enumitem}
\usepackage{amsfonts}
\usepackage{amssymb}
\usepackage{mathpazo}
\usepackage[a4paper,top=2.54cm,bottom=2.0cm,left=2.0cm,right=2.54cm]{geometry}
\usepackage{cancel}
\usepackage[usenames, svgnames,dvipsnames]{xcolor}
\usepackage{multicol}
\usepackage[T1]{fontenc}
\usepackage{calrsfs}
\usepackage{MnSymbol}
\usepackage{adjustbox}
\usepackage{lipsum}
\usepackage{listings}
\usepackage{longtable}
\usepackage{mathtools}
\usepackage{amsthm}
\usepackage[colorlinks=true, citecolor=DarkGreen, linkcolor=NavyBlue, urlcolor=DarkRed, filecolor=green]{hyperref}
\newtheorem{theorem}{Theorem}[section]

\usepackage[colorlinks=true, citecolor=DarkGreen, linkcolor=NavyBlue, urlcolor=DarkRed, filecolor=green]{hyperref}

\newtheorem{lemma}[theorem]{Lemma}

\newtheorem{obs}[theorem]{Remark}

\newtheorem{proposition}[theorem]{Proposition}

\newcommand{\F}{\mathbb{F}_q}
\newcommand{\N}{N_q(\mathcal{F})}
\newcommand{\C}{\mathcal{F}}
\newcommand{\f}{\mathcal{F}}

\usepackage{etoolbox}
\title{Frobenius nonclassicality of generalized Fermat curves with respect to conics}
\author{Nazar Arakelian and Leandro A. M. Rodrigues}

\begin{document}

\Large
\maketitle 
\normalsize

\begin{abstract}
The effective application of the St\"ohr-Voloch theory for the linear system of plane curves of a fixed degree to bound the number of rational points of a family of plane curves defined over $\F$ requires the characterization of the $\F$-Frobenius nonclassical curves in the family. In this paper, we provide necessary and sufficient conditions for certain generalized Fermat curves $\f$ defined over $\F$ to be $\F$-Frobenius nonclassical with respect to the linear system of conics. In the Frobenius classical cases, we obtain nice bounds for the number $\N$ of rational points of $\f$ via St\"ohr-Voloch theory, whereas in the Frobenius nonclassical cases, we derive explicit formulas for $\N$.
\end{abstract}

\section{Introduction}
\large

Let $\F$ be a finite field with $q=p^h$ elements, where $p$ is a prime integer, and let $\f$ be a (geometrically irreducible, projective, algebraic) curve defined over $\F$. Denote by $N_q(\f)$ the number of $\F$-rational points on a nonsingular model of $\F$. It is well known that the computation or at least the estimate of $N_q(\f)$ is one of the most compelling problems in the theory of algebraic curves defined over finite fields. Undoubtedly, the most remarkable result in this direction is the Hasse-Weil bound, see
 \cite [Theorem 5.2.3]{Sti}, which states
\begin{equation} \label{Eq1}
|N_q(\f)-(q+1)|\leq 2g\sqrt{q},
\end{equation}
where $g$ denotes the genus of $\f$. Since its inception, many authors have been trying to improve \eqref{Eq1} whenever possible. Notable contributions in this direction include the work of Serre, Stark, and Ihara \cite{Ser,Stk,Sti}.

Such efforts to determine increasingly better bounds for $N_q(\f)$ are justified, as we can find applications involving the number of rational points of an algebraic curve in other settings, such as Waring's problem over finite fields \cite{GV1}, exponential sums \cite{CP}, among others.

In 1985, Stöhr and Voloch \cite{SV} introduced a new method to estimate $N_q(\f)$. Unlike Hasse-Weil's approach, the bounds obtained via Stöhr-Voloch's method depend on the geometrical aspects of a model of $\f$. Their method not only corroborates the results obtained by Hasse-Weil, Serre and Stark, but also opens new possibilities to improve \eqref{Eq1}.

For instance, let $\f$ be a plane curve of degree $n$. 
Given an integer $s \in \{1,\ldots, n-1\}$, let $\Sigma_s$ be the linear system of curves of degree $s$ in $\mathbb{P}^2(\overline{\mathbb{F}}_q)$. Considering $M=\binom{s+2}{2}-1$, it is shown in \cite[Proposition 2.1]{SV} that there exists a sequence of integers $0=\nu_0<\ldots<\nu_{M-1}$ called the $\F$-Frobenius order sequence  of $\f$ with respect to $\Sigma_s$. Via \cite[Theorem 2.13]{SV}, we obtain
\begin{equation}\label{Eq3}
N_q(\f) \leq \dfrac{(\nu_1+\ldots+\nu_{M-1})(2g-2)+(q+M)sn}{M}.
\end{equation}
If $\nu_i=i$ for $i=0,1,\ldots,M-1$, the curve is said to be $\F$-Frobenius classical  with respect to $\sum_s$ (or with respect to curves of degree $s$). Otherwise, it is called $\F$-Frobenius nonclassical. Furthermore, given a point $Q \in \f$, if $j_0(Q),\ldots,j_M(Q)$ denote all possible intersection multiplicities of $\f$ with curves of $\Sigma_s$ at $Q$, then \eqref{Eq3} can be refined to
\begin{equation} \label{Eq4}
    N_q(\f)\leq \frac{(\nu_1+\dots+\nu_{M-1})(2g-2)+(q+M)sn-\sum A(Q)}{M},
\end{equation}
where
\begin{equation*}
    A(Q)=
    \begin{cases}
    &\sum_{i=1}^M(j_i(Q)-\nu_{i-1})-M,\: \text{if}\: Q \in \f(\F)\\
    &\sum_{i=0}^{M-1}(j_i(Q)-\nu_i),\: \text {otherwise}        
    \end{cases}.
\end{equation*}
Note that the right side of \eqref{Eq4} assumes its smaller value when $\f$ is $\F$-Frobenius classical with respect to $\Sigma_s$. As a matter of fact, bound \eqref{Eq4} improves bound \eqref{Eq1} in many instances when $\f$ is Frobenius classical, see \cite{AB1,Bor,Gar,GV1}.
However, the main obstacle for the application of the Stöhr-Voloch method is the computation of the $\F$-Frobenius order sequence of a curve, which in general  can be a difficult task. In fact, the complete classification of Frobenius nonclassical curves still an open problem even for $s=1$, i.e., the Frobenius nonclassical curves with respect to lines are not completely classified. 

So far, the characterization of Frobenius nonclassical curves is only known for certain families of curves, especially for small values of $s$, see e.g. \cite{AB2, AB1, Bor, Gar,GV1}. This classification allows us to achieve a result that serves two significant purposes, as explained in \cite{AB1,Bor}. The  $\F$-Frobenius classical curves are those that possess a better bound for the number of rational points. On the other hand, the $\F$-Frobenius nonclassical curves are those for which there potentially exist a large number of rational points.

Taking into account the Fermat curves $\mathcal{F}:ax^n+by^n=1$, \cite[Theorem 2, Theorem 3]{GV1} provide necessary and sufficient conditions under which this curve is  $\F$-Frobenius nonclassical with respect to lines (i.e., $\Sigma_1$) and conics (i.e., $\Sigma_2$). In turn, in \cite[Theorem 1.2]{AB1}, a characterization of $\F$-Frobenius nonclassical Fermat curves with respect to cubics is presented.

This paper deals with the classification of Frobenius nonclassical generalized Fermat curves, in the sense presented in \cite{AP}, with respect to lines and conics. More precisely, we consider the curves defined by $ax^n+by^m=1$ with $a,b \in \F^*$ where $m,n$ are positive integers, as well as the curves  defined by $ax^ny^m+bx^n+cy^m=1$, $a,b,c \in \F$, $c \neq -\frac{a}{b}, \:a \neq 0$ where $m,n$ are positive integers. For both curves, it will be assumed that $\min\{m,n\}>2$.

Initially, in Section \ref{pre}, we will state the concepts necessary to establish all prerequisites for conducting the proofs. A significant portion of this chapter will be based on the paper by Stöhr and Voloch \cite{SV} and the classification of $\F$-Frobenius nonclassical  curves with respect to $\Sigma_1$ and $\Sigma_2$ conducted by \cite{AB1, Gar, GV1, GV2, Par}.

Section 3, will be dedicated to the curve $\f:ax^n+by^m=1$, with $a,b \in \F^*$ where $m,n$ are positive integers. The Frobenius nonclassical curves with respect to lines of this type are already characterized, see \cite{Gar}. Thus, assuming $p>5$, we provide necessary and sufficient conditions for such curve being $\F$-Frobenius nonclassical with respect to conics. We also determine the number of $\F$-rational points in cases of $\F$-Frobenius nonclassicality. 

In Section 4 we will study curves of the type $\f:ax^ny^m+bx^n+cy^m=1$, with $a,b,c \in \F$, $ c \ne -\frac{a}{b}, \: a \ne 0$ where $m,n$ are positive integers. Curves of this type were exploited in \cite{AK} from the point of view of Finite Geometry and also in \cite{BM}, where the Stöhr-Voloch results are applied to linear systems of curves of a certain degree passing through the singularities of $\f$. Here, we show that such curves are always $\F$-Frobenius classical with respect to lines if $p>2$ and, for $p>5$ we provide a characterization of  $\F$-Frobenius nonclassical curves of this type with respect to conics. We also determine the number of $\F$-rational points for the cases of $\F$-Frobenius nonclassicality with respect to conics. 

Throughout this work, we will employ the following notation:
\begin{itemize}
    \item $\F$ denotes the finite field with $q=p^h$ elements, where $h \geq 1$ and $p$ is a prime (the characteristic of the field $\mathbb{F}_q$);
    \item $\overline{\mathbb{F}}_q$ denotes the algebraic closure of $\F$;
    \item $N_q(\f)$ stands for the number of $\F$-rational points on a nonsingular model of the curve  $\f$;
    \item $\mathbb{H}(\f)$ is the function field of $\f$ over $\mathbb{H}$, with $\f$ being an irreducible curve over $\F$ and $\mathbb{H}$ an algebraic extension of $\F$;
    \item $\Sigma_s$ refers to the linear system of all curves of degree $s$ in $\mathbb{P}^2(\overline{\mathbb{F}}_q)$;
    \item $I(P,C \cap D)$ denotes the intersection multiplicity of the plane curves $C$ and $D$ at point $P$;
    \item $D^{(r)}_tf$ is the $r$-th Hasse derivative of $f \in \overline{\mathbb{F}}_q(\f)$ with respect to the separating variable  $t$ of $ \overline{\mathbb{F}}_q(\f)$;
    \item $\mathbb{P}^n(\mathbb{K})$ is the projective space of dimension $n$ over a field $\mathbb{K}.$
\end{itemize}

\section{Preliminaries}\label{pre}

Most of the results in this section are based on \cite{AB1,GV1,HKT,SV}. The proofs will be omitted here, as they cannot only be found in the cited references, but also in many other standard sources on this subject.

Let $\mathcal{F}:f(X,Y,Z)=0$ be a projective absolutely irreducible plane curve, where $f(X,Y,Z)$ is a homogeneous irreducible polynomial of degree $n$ defined over $\F$. The function field $\F(\f)$ of $\f$ is given by $\F(x,y)$ with $f(x,y,1)=0$, where $x$ and $y$ are respectively the residues of $X/Z$ and $Y/Z$ in $\F(\f)$. In this case, at least one of the elements $x$ or $y$ is a separating element of $\F(\f)$. We alternatively can define $\mathcal{F}:f(x,y,1)=0$ meaning that $\mathcal{F}$ is the projective closure of the affine curve $f(x,y,1)=0$.

Consider the linear system $\Sigma_s$ of all projective curves of degree $s$, such that $s \in \{1,\ldots, n-1\}$. Given $P \in \mathcal{F}$, an integer $j(P)$ is called a $(\Sigma_s,P)$-order if there exists a curve $\mathcal{C}$ of degree $s$ such that $I(P,\mathcal{F} \cap \mathcal{C})=j(P)$. Denoting $j_i:=j_i(P)$, in \cite[section 1]{SV} we see that there are exactly $M+1$ $(\Sigma_s,P)$-orders, where  $M=\binom{s+2}{2}-1$. The sequence of integers $(j_0,j_1, \ldots, j_M)$ is called the sequence of $(\Sigma_s,P)$-orders.
Furthermore, $j_0=0$ and there exists an unique curve $\mathcal{H}_P$ of degree $s$, called the $s$-osculating curve to $\mathcal{F}$ at $P$, such that $I(P,\mathcal{F} \cap \mathcal{H}_P)=j_M$, see \cite[section 1]{SV}.

Let $\varphi:\mathcal{F} \rightarrow \mathbb{P}^M(\overline{\F})$ be the morphism associated to $\Sigma_s$ given by $\varphi=(\varphi_0:\ldots: \varphi_M)$, with $\varphi_i \in \F(\mathcal{F})$, and let $t$ be a local parameter at a point $P \in \f$. According to \cite[Corollary 1.3]{SV}, we have:

\begin{equation}\label{EqHiper}
\mathcal{H}_P: \det \begin{pmatrix}
X_0 &\cdots &X_M\\
(D^{(j_0)}_t\varphi_0)(P) &\cdots &(D^{(j_0)}_t\varphi_M)(P)\\
\vdots &\ddots &\vdots \\
(D^{(j_{M-1})}_t\varphi_0)(P) &\cdots &(D^{(j_{M-1})}_t\varphi_M)(P)
\end{pmatrix}
=0.
\end{equation}

From \cite[Proposition 1.4 and Theorem 1.1]{SV}, there exists a sequence of integers $\varepsilon_0<\varepsilon_1<\ldots<\varepsilon_M$ chosen minimally in lexicographic order such that the Wronskian
\begin{equation} \label{Eq1.3}
\det\left(D_\tau^{(\varepsilon_i)}\varphi_j\right)_{i,j=0,\ldots,M}
\end{equation}
is nonzero, where $\tau \in \overline{\mathbb{F}}_q(\f)$ is a separating element. Such integers $\varepsilon_i$ are called $\Sigma_s$-order of $\mathcal{F}$, and the sequence $(\varepsilon_0, \ldots, \varepsilon_M)$ is the order sequence of $\mathcal{F}$ with respect to $\Sigma_s$.

For almost every point $P$ we have that $(j_0,\ldots, j_M)=(\varepsilon_0, \ldots, \varepsilon_M)$, and such points are called $\Sigma_s$-ordinary. Points for which $(j_0,\ldots, j_M)\ne (\varepsilon_0, \ldots, \varepsilon_M)$ are called $\Sigma_s$-Weierstrass. The curve $\mathcal{F}$ is said to be classical with respect to $\Sigma_s$ when $\varepsilon_i=i$ for all $i=0,1,\ldots,M$. Otherwise, we say that $\mathcal{F}$ is nonclassical  with respect to $\Sigma_s$.
From \cite[Corollary 1.7]{SV}, we have an important result that helps us determine conditions under which a curve is classical.

\begin{proposition}\label{Prop1.1}
Let $P \in \mathcal{F}$ and $j_0, \ldots, j_M$ be the sequence of $(\Sigma_s,P)$-orders. If the integer
\begin{equation*}
    \prod \limits_{i>k}\frac{j_i-j_k}{i-k}
\end{equation*}
is not divisible by $p$, then $\mathcal{F}$ is classical with respect to $\Sigma_s$.
\end{proposition}

By \cite[Proposition 2.1]{SV} there exist integers $\nu_0,\ldots,\nu_{M-1}$, with $\nu_0<\ldots<\nu_{M-1}$ chosen minimally in lexicographic order, such that the function
\begin{equation} \label{Eq1.4}
\det \begin{pmatrix}
\varphi_0^q &\cdots &\varphi_M^q \\
D^{(v_0)}_\tau\varphi_0 &\cdots &D^{(v_0)}_\tau\varphi_M\\
\vdots &\ddots &\vdots \\
D^{(v_{M-1})}_\tau\varphi_0 &\cdots &D^{(v_{M-1})}_\tau\varphi_M
\end{pmatrix}
\end{equation}
 is non‑zero, where $\tau \in \overline{\mathbb{F}}_q(\f)$ is a separating element. Moreover,
\[
(\nu_0,\ldots,\nu_{M-1})=(\varepsilon_0,\ldots,\varepsilon_M)\setminus\{\varepsilon_I\}
\qquad\text{for some } I\in\{1,\ldots,M\}.
\]
The sequence of integers $(\nu_0,\ldots,\nu_{M-1})$ is called the $\F$‑Frobenius order sequence of $\f$ with respect to $\Sigma_s$.  If $\nu_i=i$ for every $i=0,\ldots,M-1$, we say that the curve $\f$ is $\F$‑Frobenius classical with respect to $\Sigma_s$; otherwise, we say that $\f$ is $\F$‑Frobenius nonclassical.

It is important to note that, by \cite{HKT}, we have the following.

\begin{proposition}\label{Prop1.4}
Let $p>M$.  If $\f$ is $\F$‑Frobenius nonclassical with respect to $\Sigma_s$, then $\f$ is nonclassical with respect to $\Sigma_s$.
\end{proposition}

Now we state the main results on the nonclassicality and $\F$‑Frobenius nonclassicality of $\mathcal{F}:aX^n+bY^n=Z^n$ (Fermat Curve) with respect to $\Sigma_1$ and $\Sigma_2$. We assume that $n$ is not divisible by $p$. We begin with the nonclassicality of $\mathcal{F}$ with respect to $\Sigma_1$, developed in \cite{Par}.

\begin{proposition}\label{Propretas}
Assuming $p\ne 2$, $\mathcal{F}$ is nonclassical with respect to $\Sigma_1$ if and only if
$$n \equiv 1 \: (\mathrm{mod}\,p).$$
\end{proposition}

According to \cite[Theorem 2]{GV1} and \cite{Bor}, we have the following result on the $\F$‑Frobenius nonclassicality of $\mathcal{F}$ with respect to $\Sigma_1$.

\begin{theorem}\label{Theoretas}
$\mathcal{F}$ is $\F$‑Frobenius nonclassical with respect to $\Sigma_1$ if and only if 
$$n=\frac{q-1}{p^r-1}$$
for some integer $r<h$, with $r\mid h$, and $a,b \in \mathbb{F}_{p^r}$.
\end{theorem}

From \cite[Section 1]{Gar} and \cite{Bor}, for a curve $\mathcal{C}: aX^n+bY^mZ^{n-m}=Z^{n}$ with $p \nmid mn$ and $n \geq m$ 
we obtain:

\begin{theorem}\label{Theoretas1}
$\mathcal{C}$ is $\F$‑Frobenius nonclassical with respect to $\Sigma_1$ if and only if  $$n=m=\frac{q-1}{p^r-1}$$
for some integer $r<h$, with $r\mid h$, and $a,b \in \mathbb{F}_{p^r}^{*}$.
\end{theorem}

We now recall the results for the linear system of conics $(\Sigma_2)$. By \cite[Theorem 3]{GV2} we have:

\begin{proposition}\label{PropConicas}
For $p \ge 7$, $\mathcal{F}$ is nonclassical with respect to $\Sigma_2$ if and only if 
$$p\mid (n-1)(n-2)(n+1)(2n-1).$$
\end{proposition}

Concerning $\F$‑Frobenius nonclassicality, by \cite[Theorem 3]{GV1} and \cite[Appendix B]{AB1}:

\begin{theorem}\label{Theoconicas}
For $p \ge 7$, $\mathcal{F}$ is $\F$‑Frobenius nonclassical with respect to $\Sigma_2$ precisely in one of the following cases:
\begin{itemize}
  \item[(1)] $p \mid (n-1);$
  \item[(2)] $p \mid (n-2)$ and $n=\dfrac{2(q-1)}{p^r-1}$ with $r<h$, $r\mid h$, and $a,b \in \mathbb{F}_{p^r};$
  \item[(3)] $p \mid (2n-1)$ and $n=\dfrac{q-1}{2(p^r-1)}$ with $r<h$, $r\mid h$, and $a^2,b^2 \in \mathbb{F}_{p^r};$
  \item[(4)] $q=n+1$ and $a+b=1.$
\end{itemize}
\end{theorem}

In order conclude the classification of nonclassical curves of type $\f:x^n+y^m=1$ and $\f:ax^ny^m+x^n+y^m=1$ with respect to $\Sigma_2$, we state the following lemma, which is an adaptation of \cite[Lemma 3.4]{AB1} to our case of conics. Since the proof is analogous to the proof of \cite[Lemma 3.4]{AB1}, it will be omitted.

\begin{lemma}\label{Le2.1}
Assume $p>5$. Let $\overline{\mathbb{F}}_q(x,y)$ be the function field of $\f$, and let $P=(u:v:1) \in \f$ be a general point. Suppose there exists a polynomial 
\[
G(X,Y) = \sum a_{ij}(x,y)^p X^i Y^j \in \overline{\mathbb{F}}_q[x,y][X,Y]
\]
of degree $d \geq 2$ such that $G(x,y)=0$. For 
\[
G_P(X,Y) := \sum a_{ij}(u,v)^p X^i Y^j \in \overline{\mathbb{F}}_q[X,Y],
\]
the following holds:

\begin{itemize}
    \item[(a)]\label{Le2.1a} If $G_P(X,Y)$ is irreducible of degree $d=2$, then $\f$ is nonclassical with respect to $\Sigma_2$, and the curve $\mathcal{G}_P: G_P(X,Y)=0$ is the osculating conic to $\f$ at $P$.
    
    \item[(b)]\label{Le2.1b} If the curve $\mathcal{G}_P: G_P(X,Y)=0$ satisfies $I(P, \mathcal{G}_P \cap \mathcal{C})<p$ for every conic $\mathcal{C}$, then $\f$ is classical with respect to $\Sigma_2$.
\end{itemize}
\end{lemma}

\begin{obs}\label{Ob2.1}
Note that if $\mathcal{G}_P$ is irreducible of degree $d$, with $2 < d < \dfrac{p}{2}$, then by Bézout's Theorem, the conditions of Lemma~\ref{Le2.1b}(b) are satisfied; that is, $\f$ is classical with respect to $\Sigma_2$.
\end{obs}

In what follows in the paper, we use ``w.r.t.'' to mean “with respect to.''

\section{Curve $\f:ax^n+by^m=1$}

Consider the curve $\mathcal{F}$ defined over the finite field $\F$ as the projective closure of the affine curve $ax^n+by^m=1$, where $a,b \in \F$ and $m,n \in \mathbb{Z}_+$.
Henceforth, we assume without loss of generality that $n \geq m>2$. When $p>2$, the $\F$-Frobenius nonclassicality w.r.t. $\Sigma_1$ is established in Theorem \ref{Theoretas1}. In this section, we establish necessary and sufficient conditions for $\f$ to be $\F$-Frobenius nonclassical w.r.t. $\Sigma_2$. In order to obtain such a characterization, in view of Proposition \ref{Prop1.4} we start by studying the nonclassicality of $\f$ w.r.t. $\Sigma_2$. One should note that regarding $x$ and $y$ as elements in the function field of $\f$, both are separating elements over $\F$.

\subsection{Classicality of $\f$ with respect to $\Sigma_2$}

Since the classicality of the curve is a geometric property, in this subsection we may assume $a=b=1$.

\begin{proposition}\label{Prop2.2}
Let $p>5$, and let $\f:x^n+y^m=1$ be defined over $\F$. If $\f$ is nonclassical w.r.t. $\sum_2$, then the following hold:  
$$p \mid (2n-1)(n+1)(n-2)(n-1) \quad \text{and} \quad p \mid (2m-1)(m+1)(m-2)(m-1).$$
\end{proposition}
\begin{proof}
Consider the affine points of $\f$ given by $P_{\xi}=(0, \xi)$, where $\xi$ is an $m$-th root of unity; $P_{\rho}=(\rho, 0)$, where $\rho$ is an $n$-th root of unity. The tangent line to $\f$ at $P_\xi$ is $\ell_{P_\xi}:y=\xi$, and one can check that $I(P_\xi,\f \cap \ell_{P_\xi})=n$. Thus the $(\Sigma_1,P_{\xi})$-order sequence is given by $(0,1,n)$. Since $n \geq m>2$, considering the conics given by the union of two of these lines, we obtain that the $(\Sigma_2,P_{\xi})$-order sequence is $(0,1,2,n,n+1,2n)$. Analogously, the $(\Sigma_2,P_{\rho})$-order sequence is given by $(0,1,2,m,m+1,2m)$. The result then follows from Proposition \ref{Prop1.1}.
\end{proof}

\begin{proposition}\label{Prop2.3}
Assume that $p>5$. The curve $\f$ is classical w.r.t. $\Sigma_2$ in the following cases:
\begin{itemize}
    \item[(a)] $p \mid (2n-1)$ and $p \mid (m+1)$;
    \item[(b)] $p \mid (2m-1)$ and $p \mid (n+1)$;
    \item[(c)] $p \mid (2n-1)$ and $p \mid (m-2)$;
    \item[(d)] $p \mid (2m-1)$ and $p \mid (n-2)$;
    \item[(e)] $p \mid (m+1)$ and $p \mid (n-2)$;
    \item[(f)] $p \mid (n+1)$ and $p \mid (m-2)$.
\end{itemize}
\end{proposition}

\begin{proof}
    For case (a), consider integers $r,s,k,l$ such that $2n = p^r k + 1$ and $m = p^s l - 1$, with $p \nmid k$ and $p \nmid l$. Without loss of generality, assume $r \geq s$, i.e., $r = s + d$ for some integer $d\geq 0$, so that $2n = (p^d k)p^s + 1 = w p^s + 1$. We have in the function field $\overline{\mathbb{F}}_q(\f)$:
\begin{equation*}
(x^n + y^m - 1)(x^n - y^m + 1) = 0 \iff 
x^{2n} - y^{2m} + 2y^m - 1 = 0
\end{equation*}
\begin{equation}\label{E.11}
\iff (x^w)^{p^s} x - ((y^l)^2)^{p^s} y^{-2} + 2 (y^l)^{p^s} y^{-1} - 1 = 0.
\end{equation}

Let $P=(u:v:1) \in \f$ be a general point, and set $\alpha = (u^w)^{p^s}$ and $\beta = (v^l)^{p^s}$. 
Consider the cubic
 
\[
\mathcal{G}_P: \alpha XY^2 - \beta^2 Z^3 + 2 \beta YZ^2 - Y^2 Z = 0
\]  
One can check directly that $(1:0:0)$ is the only singular point of $\mathcal{G}_P$. By de-homogenizing the above polynomial with respect to the variable $X$ we obtain $\alpha Y^2-Z(\beta Z-Y)^2$. Since $P$ is a general point, we may assume $\beta \neq 0$, and then $Y \nmid Z(\beta Z-Y)^2$, whence   $\mathcal{G}_P$ is absolutely irreducible.
Thus Remark~\ref{Ob2.1} together with \eqref{E.11} imply that $\f$ is classical w.r.t. $\Sigma_2$. Case (b) is analogous to case (a).

Now we tackle case (c): consider integers $r,s,k,l$ such that $2n = p^r k + 1$ and $m = p^s l + 2$, with $p \nmid k$ and $p \nmid l$. In this case we may assume $r \geq s$, i.e., $r = s + d$ for some integer $d \geq 0$, so that $2n = (p^d k)p^s + 1 = w p^s + 1$. Thus
 \begin{equation*}
    (x^n+y^m-1)(x^n-y^m+1)=0 \iff 
    x^{2n}-y^{2m}+2y^m-1=0
    \end{equation*}
    \begin{equation}\label{E2.12}
    \iff(x^{w})^{p^s}x-((y^l)^2)^{p^s}y^{4}+2(y^l)^{p^s}y^2-1=0.
    \end{equation}
Hence, for a general point $P=(u:v:1) \in \f$, set $\alpha = (u^w)^{p^s}$ and $\beta = (v^l)^{p^s}$. The quartic \\$\mathcal{G}_P: \alpha XZ^3- \beta ^2 Y^4+2 \beta Y^2Z^2-Z^4=0$ is absolutely irreducible. Indeed, $(1:0:0)$ is the only singular point of $\mathcal{G}_P$, and it has multiplicity  $3$. Proceeding as we did with the cubic in the previous case, we obtain the result. Therefore the conclusion follows again from Remark ~\ref{Ob2.1} provided that $p>7$. If $p=7$, set $\mathcal{I}:XZ^3-Y^4+2Y^2Z^2-Z^4=0$. Then $\mathcal{I}$ is projectively equivalent to $\mathcal{G}_P$. One can show that if $Q=(2:s:1), \: s^2=-2$, then $I(Q,\mathcal{I}\cap\ell_Q)=3$ where $\ell_Q$ is the tangent line to $\mathcal{I}$ at $Q$. This together with \cite[Lemma 3.3]{AB1} and B\'ezout`s Theorem imply that $I(Q,\mathcal{I}\cap\mathcal{C})< 7$ for all conic $\mathcal{C}$; in particular,  $\mathcal{G}_P$ is classical w.r.t. $\Sigma_2$. The result then follows from Lemma \ref{Le2.1b}. Case (d) is analogous to (c). 

Finally, assume that (e) holds (case (f) is analogous to case (e)). Let  $r,s,k,l$  be integers such that $n=p^rk+2$ and $m=p^sl-1$, $p\nmid k$, $p\nmid l$. Assuming again that $r \geq s$, and thus $r= s+d$ for some $d\geq 0$,  we obtain $n=(p^dk)p^s+2=wp^s+2$. Hence
    \begin{equation}\label{E2.13}
    x^n+y^m-1=0 \iff (x^w)^{p^s}x^2+(y^l)^{p^s}y^{-1}-1=0.
    \end{equation}
This leads us to the cubic $\mathcal{G}_P: \alpha X^2Y-YZ^2+\beta Z^3=0$, where $\alpha=(u^w)^{p^s}$ and $\beta=(v^l)^{p^s}$, with $P=(u:v:1) \in \f$ being a general point. Such a cubic is absolutely irreducible, since it has only one singular point $(0:1:0)$, which is a node. The result follows as in the previous cases.

\end{proof}

\begin{proposition}\label{Prop2.4}
Assume $p>5$. The curve $\f$ is nonclassical w.r.t. $\Sigma_2$ in the following cases:
\begin{itemize}
    \item[(a)] $p \mid (m-2)$ and $p \mid (n-2)$;
    \item[(b)] $p \mid (m+1)$ and $p \mid (n+1)$;
    \item[(c)] $p \mid (2m-1)$ and $p \mid (2n-1)$;
    \item[(d)] $p \mid (m-1)$ and $p \mid (n-1)$;
    \item[(e)] $p \mid (2m-1)$ and $p \mid (n-1)$;
    \item[(f)] $p \mid (m-1)$ and $p \mid (2n-1)$;
    \item[(g)] $p \mid (m-2)$ and $p \mid (n-1)$;
    \item[(h)] $p \mid (m-1)$ and $p \mid (n-2)$;
    \item[(i)] $p \mid (m+1)$ and $p \mid (n-1)$;
    \item[(j)] $p \mid (m-1)$ and $p \mid (n+1)$.
\end{itemize}
\end{proposition}

\begin{proof}
The proof of each item will follow from Lemma~\ref{Le2.1a}(a). Following the notation of such Lemma, in each case we provide a polynomial $G(X,Y)\in \overline{\mathbb{F}}_q[x,y][X,Y]$ and for a general point $P=(u:v:1)$, the corresponding absolutely irreducible conic $\mathcal{G}_P$. In what follows, for integers $r,s,w,l$, we set $\alpha=(u^w)^{p^s}$ and $\beta=(v^l)^{p^s}$.
\begin{itemize}
\item[(xa)] Consider integers $r,s,k,l$ such that $n=p^rk+2$ and $m=p^sl+2$, $p\nmid k$, $p\nmid l$. There is no loss of generality in assuming that $r \geq s$. Thus  $r= s+d$ for some $d\geq 0$ and $n=(p^dk)p^s+2=wp^s+2$. Hence
    \begin{equation}\label{E2.14}
    x^n+y^m-1=0 \iff (x^w)^{p^s}x^2+(y^l)^{p^s}y^2-1=0.
    \end{equation}
    By \eqref{E2.14} we have the associated irreducible conic  $\mathcal{G}_P: \alpha X^2+\beta Y^{2}-Z^2=0.$ Therefore, from Lemma \ref{Le2.1a} (a) we conclude that $\f$ is nonclassical w.r.t. $\Sigma_2$.
\item[(b)] Here we take $r,s,k,l$ such that $n=p^rk-1$ and $m=p^sl-1$, $p\nmid k$, $p\nmid l$ and WLOG $r=s+d$, with $d \geq 0$. Thus  $n=(p^dk)p^s-1=wp^s-1$, and we have
    \begin{equation}\label{E2.15}
    (x^n+y^m-1)(xy)=0 \iff (x^w)^{p^s}y+(y^l)^{p^s}x-xy=0
    \end{equation}
     and the associated irreducible conic $\mathcal{G}_P:\alpha YZ +\beta XZ -XY=0.$
\item[(c)] Consider $r,s,k,l$ such that $2n=p^rk+1$ and $2m=p^sl+1$, $p\nmid k$, $p\nmid l$ and WLOG $r=s+d$ with $d \geq 0$. Thus $2n=(p^dk)p^s+1=wp^s+1$ and 
\begin{equation}\label{E2.16}
    \begin{split}
    &(x^n+y^m)^2=1 \Longrightarrow\\
    &x^{2n}+2x^{n}y^{m}+y^{2m}=1 \Longrightarrow\\
    &(x^{2n}+y^{2m}-1)^2=(2x^{n}y^m)^2\Longrightarrow \\
    &x^{4n}+y^{4m}+1-2x^{2n}y^{2m}-2x^{2n}-2y^{2m}=0\Longrightarrow\\
    &((x^w)^{p^{s}})^2x^2+((y^l)^{p^s})^2y^2+1-2(x^w)^{p^s}(y^l)^{p^s}xy-2(x^w)^{p^s}x-2(y^l)^{p^s}y=0.
     \end{split}
    \end{equation}
The associated irreducible conic in this case is $\mathcal{G}_P: \alpha^2 X^2+\beta^2 Y^{2}+Z^2-2\alpha \beta XY - 2\alpha XZ - 2 \beta YZ=0.$
\item[(d)] From \cite[Theorem 1]{Gar}, $\f$ is nonclassical w.r.t. $\Sigma_1$ if and only if $p \mid m-1$ and $p \mid n-1$. Thus since $p>5$, considering degenerate conics we conclude that $\f$ is nonclassical w.r.t. $\Sigma_2$.
\item[(e)] and (f) Here we consider  $r,s,k,l$ such that $n=p^rk+1$ and $2m=p^sl+1$, $p\nmid k$, $p\nmid l$. Assume WLOG that $r=s+d$ with $d \geq 0$. Thus $n=(p^dk)p^s+1=wp^s+1$ and we have
\begin{equation*}
    (x^n+y^m-1)(-x^n+y^m+1)=0
   \iff y^{2m}-x^{2n}+2x^n-1-0
   \end{equation*}
   \begin{equation}\label{E2.17}
   \iff (y^l)^{p^s}y-(x^{2w})^{p^s}x^2+2(x^k)^{p^s}x-1=0.
    \end{equation}
The associated irreducible conic in this case is $\mathcal{G}_P: \beta YZ -\alpha^2X^2+2\alpha XZ-Z^2=0.$ Since we never use the fact that $n \geq m$, the case (f) is analogous.

\item[(g)] and (h) With the notation as in the previous items, here we have
\begin{equation}\label{E0}
    (x^n+y^m-1)=0  
   \iff (x^{w})^
 {p^s}x^1+(y^l)^{p^s}y^2-1=0
    \end{equation}
and the associated irreducible conic $\mathcal{G}_P:\alpha XZ +\beta Y^2-Z^2=0.$ Case (h) is analogous.

\item[(i)] and (j)  With the notation as in the previous items, here we have
\begin{equation}\label{E1}
    (x^n+y^m-1) = 0 
   \iff (x^{w})^
 {p^s}x+(y^l)^{p^s}y^{-1}-1=0 \iff 
 (x^{w})^{p^s}xy+(y^l)^{p^s}-y=0
    \end{equation}
and the associated irreducible conic $\mathcal{G}_P:\alpha XY +\beta Z^2 -YZ=0.$ Again, item (j) is analogous.
\end{itemize}
\end{proof}

As an immediate consequence of Propositions~\ref{Prop2.2}, \ref{Prop2.3}, and \ref{Prop2.4}, we have the following classification.

\begin{theorem}\label{Teo2.3}
Assume $p>5$. The curve $\f$ is nonclassical w.r.t. $\Sigma_2$ if and only if one of the following conditions holds:
\begin{itemize}
    \item[(a)] $p \mid (m-2)$ and $p \mid (n-2)$;
    \item[(b)] $p \mid (m+1)$ and $p \mid (n+1)$;
    \item[(c)] $p \mid (2m-1)$ and $p \mid (2n-1)$;
    \item[(d)] $p \mid (m-1)$ and $p \mid (n-1)(n-2)(2n-1)(n+1)$;
    \item[(e)] $p \mid (n-1)$ and $p \mid (m+1)(m-2)(2m-1)$.
\end{itemize}
\end{theorem}

\subsection{$\F$-Frobenius classicality of $\f$ w.r.t. $\Sigma_2$}

\quad The aim of this section is to give a characterization of the curves $\f:ax^n+by^m=1$ that are $\F$-Frobenius nonclassical w.r.t. $\Sigma_2$ for $p>5$. Hence, in what follows we always assume $p>5$, and $\mathcal{H}_P$ will denote the osculating conic to $\f$ at $P\in \f$.  

In view of Proposition \ref{Prop1.4}, we may restrict our attention to the cases in which $\f$ is nonclassical for $\Sigma_2$.

\begin{proposition} \label{Prop2.5}
Assume that $p>5$, $\f$ is classical w.r.t. $\Sigma_1$ and nonclassical w.r.t. $\Sigma_2$. For $P=(u:v:1) \in \f$, $uv \neq 0$, the osculating conic $\mathcal{H}_P$ to $\f$ at $P$ is the irreducible projective curve defined by $H_P(X,Y,Z)=0$ where

\[
H_P(X,Y,Z) =
\begin{cases}
au^{n-2}X^2 + bv^{m-2}Y^2 - Z^2 ,  \text{ if } p \mid (m-2) \text{ and } p \mid (n-2); \\[2mm] 
au^{n+1} YZ + bv^{m+1} XZ - XY ,  \text{ if } p \mid (m+1) \text{ and } p \mid (n+1);\\[2mm]
a^4 u^{4n-2} X^2 + b^4 v^{4m-2} Y^2 + Z^2 - 2a^2b^2 u^{2n-1} v^{2m-1} XY - 2a^2 u^{2n-1} XZ- \\[2mm]\qquad \qquad \qquad  2b^2 v^{2m-1} YZ ,  \text{ if } p \mid (2m-1) \text{ and } p \mid (2n-1);\\[2mm]
b^2 v^{2m-1} YZ - a^2 u^{2n-2} X^2 + 2 a u^{n-1} XZ - Z^2 ,  \text{ if } p \mid (2m-1) \text{ and } p \mid (n-1);\\[2mm]
a u^{n-1} XY + b v^{m+1} Z^2 - YZ ,  \text{ if } p \mid (m+1) \text{ and } p \mid (n-1);\\[2mm]
a u^{n-1} XZ + b v^{m-2} Y^2 - Z^2 ,  \text{ if } p \mid (m-2) \text{ and } p \mid (n-1);\\[2mm]
a u^{n-2} X^2 + b v^{m-1} YZ - Z^2, \text{ if } p \mid (m-1) \text{ and } p \mid (n-2);\\[2mm]
a^2 u^{2n-1} XZ - b^2 v^{2m-2} Y^2 + 2 b v^{m-1} YZ - Z^2 ,  \text{if } p \mid (m-1) \text{ and } p \mid (2n-1);\\[2mm]
b v^{m-1} XY + a u^{n+1} Z^2 - XZ, \text{ if } p \mid (m-1) \text{ and } p \mid (n+1).
\end{cases}
\]
\end{proposition}

\begin{proof}
This follows from Proposition~\ref{Prop2.4} and Lemma \ref{Le2.1a}.
\end{proof}

Proposition \ref{Prop2.5} assures that if $\f$ is nonclassical w.r.t. $\Sigma_2$, then the osculating conic at a general point of $\f$ is irreducible. With this information on hands, we are able to prove the following result. The proof will be omitted since it is analogous to the one of \cite[Lemma 4.1]{AB1}.

\begin{proposition} \label{Prop2.6}
Assume $p>5$. If $\f$ is classical w.r.t. $\Sigma_1$ and nonclassical w.r.t. $\Sigma_2$, then the following statements hold:
\begin{itemize}
    \item[(a)] The order sequence of $\f$ w.r.t. $\Sigma_2$ is $(0,1,2,3,4,p^r)$, for some $r>0$.
    \item[(b)] The curve $\f$ is $\F$-Frobenius nonclassical w.r.t. $\Sigma_2$ if and only if $\Phi_q(P) \in \mathcal{H}_P$ 
     for infinitely many points $P \in \f$,  where $\Phi_q:\f \longrightarrow \f$ denotes the $\F$-Frobenius map.
\end{itemize}
\end{proposition}

\begin{proposition}\label{Prop2.1'}
Let $p>5$ and assume that $p \mid (n-2)$ and $p \mid (m-2)$. Then, the curve $\f$ defined over $\F$ is $\F$-Frobenius nonclassical w.r.t. $\Sigma_2$ if and only if 
\[
m = n = \frac{2(q-1)}{p^r-1}
\] 
for some $r < h$ such that $r \mid h$ and $a,b \in \mathbb{F}_{p^r}$.
\end{proposition}
\begin{proof}
    By Propositions~\ref{Prop2.5} and \ref{Prop2.6}(b), the $\F$-Frobenius nonclassicality of $\f$ is equivalent to the function 
\begin{equation}\label{E1}
ax^{n-2+2q}+by^{m-2+2q}-1
\end{equation}
being vanishing. We can write $n-2+2q=p^r \alpha$ and $m-2+2q=p^s \beta$, where $r,s>0$ and both $\alpha$ and $\beta$ are co-prime with $p$. If $r<s$, \eqref{E1} implies 
$$
a^{1/p^r}x^\alpha+b^{1/p^r}y^{p^{s-r}\beta}=1,
$$
which gives that $x$ is not a separating element in $\F(\f)$, a contradiction. The assumption $s<r$ leads to a contradiction as well. Thus $r=s$, and \eqref{E1} is equivalent to
\begin{equation}\label{E1.1}
a^{1/p^r}x^\alpha+b^{1/p^r}y^{\beta}-1
\end{equation}
being vanishing in $\F(\f)$, which means that \eqref{E1.1} is divisible by $ax^n+by^m-1$. Since  \eqref{E1.1} is absolutely irreducible, it vanishes if and only if $\alpha=n$ and $\beta=m$. Therefore
$$
m=n=\frac{2(q-1)}{p^r-1},
$$
and the result then follows from Theorem~\ref{Theoconicas}
\end{proof}

\begin{proposition}\label{Prop2.2'}
Let $p>5$ and suppose that $p \mid (n+1)$ and $p \mid (m+1)$. Then, the curve $\f$ defined over $\F$ is $\F$-Frobenius nonclassical w.r.t. $\Sigma_2$ if and only if $m=n$, $q=n+1$, and $a+b=1.$
\end{proposition}
\begin{proof}
    By Propositions \ref{Prop2.5} and \ref{Prop2.6} (b), we have that $\f$ is $\F$-Frobenius nonclassical w.r.t. $\Sigma_2$ if and only if
\begin{equation}\label{E2.3'}
   ax^{n+1}y^q+by^{m+1}x^q-x^qy^q=0.
\end{equation}

Let us first assume that $n+1-q < 0$. In this case, we can rewrite \eqref{E2.3'} as
\begin{equation*}
   ay^{q-m-1}+bx^{q-n-1}-x^{q-n-1}y^{q-m-1}=0.
\end{equation*}

This means that
\begin{equation} \label{E1m}
   ay^{q-m-1}+bx^{q-n-1}-x^{q-n-1}y^{q-m-1} 
   =(ax^n+by^m-1)h(x,y)        
\end{equation}
for some $h(x,y) \in \mathbb{F}_q[x,y]\setminus \{0\}$. Substituting $x=0$ into \eqref{E1m}, we obtain
\begin{equation*}
   ay^{q-m-1}=(by^m-1)h(0,y),
\end{equation*}
which yields a contradiction.  

Thus, we may assume that $n+1-q \geq 0$, and in this case \eqref{E2.3'} can be rewritten as
\begin{equation} \label{E3}
   ax^{n+1}y^q+by^{m+1}x^q-x^qy^q=0 
   \iff ax^{\,n+1-q} +by^{m+1-q}-1=0.
\end{equation}

Using an argument analogous to that in the proof of Proposition \ref{Prop2.1'}, we conclude that \eqref{E3} is equivalent to $m=n$. The result then follows from Theorem~\ref{Theoconicas}.
\end{proof}

\begin{proposition}\label{Prop2.3'}
Let $p>5$ and suppose that $p \mid (2n-1)$ and $p \mid (2m-1)$. Then, the curve $\f$ defined over $\F$ is $\F$-Frobenius nonclassical w.r.t. $\Sigma_2$ if and only if 
\[ 
   m=n=\frac{q-1}{2(p^r-1)} 
\] 
for some $r<h$ with $r \mid h$, and $a^2,b^2 \in \mathbb{F}_{p^r}$.
\end{proposition}
\begin{proof}
Assume that $\f$ is $\F$-Frobenius nonclassical w.r.t. $\Sigma_2$.From Propositions \ref{Prop2.5} and \ref{Prop2.6} (b), we have
\begin{equation}\label{aux2m-1}
   a^4x^{4n-2+2q}+b^4y^{4m-2+2q}+1
   -2a^2b^2x^{2n-1+q}y^{2m-1+q}
   -2a^2x^{2n-1+q}
   -2b^2y^{2m-1+q}=0.
\end{equation}
This means that 
\begin{equation}\label{Eq2.8'}
       \left(b^2y^{2m-1+q}-1\right)^2+a^2x^{2n-1+q}\left(a^2x^{2n-1+q}-2b^2y^{2m-1+q}-2\right)=(ax^n+by^m-1)h(x,y),  
    \end{equation}
    for some $h(x,y) \in \F[x,y]\setminus\{0\}$. Evaluating both sides of \eqref{Eq2.8'} at $x=0$ yields
    \begin{equation*}
        \left(b^2y^{2m-1+q}-1\right)^2=(by^m-1)h(0,y)
    \end{equation*}
    and thus $m\mid 2m-1+q$. In particular $m|q-1$. Replacing $y^m$ with $(1-ax^n)/b$ in \eqref{aux2m-1} we obtain
    $$
    \left(b^2\left(\frac{1-ax^n}{b} \right)^{2+\frac{q-1}{m}}-1\right)^2=-a^2x^{2n-1+q}\left(a^2x^{2n-1+q}-2b^2\left(\frac{1-ax^n}{b} \right)^{2+\frac{q-1}{m}}-2\right).
    $$
Comparing the degrees of both sides of the equation above, we conclude that $m=n$, and the result again is a consequence of Theorem~\ref{Theoconicas}. 
\end{proof}

\begin{proposition}\label{Prop2.7'}
Let $p>2$ and suppose that $p \mid (n-1)$ and $p \mid (m-1)$. Then, the curve $\f$ defined over $\F$ is $\F$-Frobenius nonclassical w.r.t. $\Sigma_2$.
\end{proposition}

\begin{proof}
From \cite[Theorem 1]{Gar}, $\f$ is nonclassical w.r.t. $\Sigma_1$. Hence there exists $r>0$ such that the order sequence of $\f$ w.r.t $\Sigma_1$ is $(0,1,p^r)$. Considering degenerate conics, we conclude that the order sequence of $\f$ w.r.t $\Sigma_2$ is $(0,1,2,p^r,p^r+1,2p^r)$. The conclusion follows from the fact that the $\F$-Frobenius order sequence w.r.t. $\Sigma_2$ is a subsequence of $(0,1,2,p^r,p^r+1,2p^r)$.
\end{proof}

\begin{proposition}\label{Prop2.8'}
Let $p>5$ and suppose that $p \mid (n-2)$ and $p \mid (m-1)$. Then, the curve $\f$ defined over $\F$ is $\F$-Frobenius nonclassical w.r.t. $\Sigma_2$ if and only if 
\[
   n=2m=\frac{2(q-1)}{p^r-1}
\] 
for some $r<h$ with $r \mid h$ and $a,b \in \mathbb{F}_{p^r}$.
\end{proposition}

\begin{proof}
By Propositions \ref{Prop2.5} and \ref{Prop2.6} (b), we have that $\f$ is $\F$-Frobenius nonclassical w.r.t. $\Sigma_2$ if and only if
\begin{equation} \label{E2.26}
   ax^{n-2+2q}+by^{m-1+q}-1=0.
\end{equation}
Let $p^r$ be the highest power of $p$ that divides $m-1+q$. Since both $x$ and $y$ are separating elements, we have that $p^r$ is the highest power of $p$ dividing $n-2+2q$ as well. Hence from \eqref{E2.26} we have 
\begin{equation} \label{E2.26-2}
   a^{1/p^r}x^{\frac{n-2+2q}{p^r}}+b^{1/p^r}y^{\frac{m-1+q}{p^r}}-1=0.
\end{equation}
Note that since the both $(n-2+2q)/p^r$ and $(m-1+q)/p^r$ are not divisible by $p$, the left side of \eqref{E2.26-2} is absolutely irreducible. Thus \eqref{E2.26-2} is equivalent to 
$$
a^{1/p^r}x^{\frac{n-2+2q}{p^r}}+b^{1/p^r}y^{\frac{m-1+q}{p^r}}-1=\lambda(ax^n+by^m-1),
$$
where $\lambda \in \overline{\mathbb{F}}_q$. This is equivalent to $\lambda=1$, $n-2+2q=np^r$, $m-1+q=mp^r$, $a=a^{1/p^r}$ and $b=b^{1/p^r}$.
\end{proof}

\begin{proposition} \label{Prop2.9'}
Let $p>5$ and suppose that $p \mid (2n-1)$ and $p \mid (m-1)$. Then the curve $\f$ defined over $\F$ is $\F$-Frobenius classical w.r.t. $\Sigma_2$.
\end{proposition}
\begin{proof}
   Again by Propositions \ref{Prop2.5} and \ref{Prop2.6} (b), $\f$ is $\F$-Frobenius nonclassical w.r.t. $\Sigma_2$ if and only if
\begin{equation}\label{E2.28}
   a^2x^{2n-1+q} - b^2y^{2m-2+2q} + 2b y^{m-1+q} - 1 = 0 \Longleftrightarrow \\
   (ax^{n+\frac{q-1}{2}})^2-(by^{m-1+q}-1)^2=0.
\end{equation}
The last equality is equivalent to 
$$
a^{1/p^r}x^{\frac{2n+q-1}{2p^r}} \pm (b^{1/p^r}y^\frac{m+q-1}{p^r}-1)=0,
$$
where $p^r$ is the common highest power of $p$ dividing both $2n+q-1$ and $m+q-1$. Since the left side of the last equality is irreducible in both $\pm$ cases, we obtain that $(2n+q-1)/2p^r=n$ and $(m+q-1)/p^r=m$, and then $2n=m$. This is a contradiction since we are assuming that $n \geq m$.
\end{proof}

\begin{proposition}\label{Prop2.10'}
Let $p>5$ and suppose that $p \mid (n+1)$ and $p \mid (m-1)$. Then the curve $\f$ defined over $\F$ is $\F$-Frobenius classical w.r.t. $\Sigma_2$.
\end{proposition}
\begin{proof}
   
By Propositions \ref{Prop2.5} and \ref{Prop2.6} (b), $\F$ is $\F$-Frobenius nonclassical w.r.t. $\Sigma_2$ if and only if
\begin{equation}\label{E2.30}
   b y^{m-1+q} x^q + a x^{n+1} - x^q = 0.
\end{equation}
An analogous argument to that in Proposition \ref{Prop2.2'} shows that the case $n+1<q $ is impossible. On the other hand, if $n+1 \geq q$, we obtain $by^{m-1+q}+ax^{n+1-q}=1$, which has degree $<n$ in $x$, leading to a contradiction.
\end{proof}

\begin{proposition}\label{Prop2.4'}
Let $p>5$ and suppose that $p \mid (n-1)$ and $p \mid (2m-1)$. Then the curve $\f$ defined over $\F$ is $\F$-Frobenius nonclassical w.r.t. $\Sigma_2$ if and only if
\[
   2m = n = \frac{q-1}{p^r-1}
\]
for some $r<h$ where either $r \mid h$ and $a,b \in \mathbb{F}_{p^r}$ or $2r \mid h$, $a \in \mathbb{F}_{p^r}$ and $b \in \mathbb{F}_{p^{2r}}$ with $b^{p^r}=-b$.
\end{proposition}
\begin{proof}
    Here by Propositions \ref{Prop2.5} and \ref{Prop2.6} (b), the  $\F$-Frobenius nonclassicality of $\f$ is equivalent to
\begin{equation}\label{E2.22}
   b^2y^{2m-1+q} - a^2x^{2n-2+2q} + 2a x^{n-1+q} - 1 = 0 \Longleftrightarrow (by^{m+\frac{q-1}{2}})^2-(ax^{n+q-1}-1)^2=0,
\end{equation}
which in turn is equivalent to 
\begin{equation}\label{E2.22-2}
   b^{1/p^r}y^{\frac{2m+q-1}{2p^r}} \pm (a^{1/p^r}x^\frac{n+q-1}{p^r}-1)=0
\end{equation}
for some $r>0$. In any case, the left side of \eqref{E2.22-2} is irreducible, and we immediately conclude that $2m=n=\frac{q-1}{p^r-1}$. Now, when $b^{1/p^r}y^{\frac{2m+q-1}{2p^r}} +(a^{1/p^r}x^\frac{n+q-1}{p^r}-1)$ is a factor of $ax^n+by^m-1$, then $a,b \in \mathbb{F}_{p^r}$. If $b^{1/p^r}y^{\frac{2m+q-1}{2p^r}} -(a^{1/p^r}x^\frac{n+q-1}{p^r}-1)$ is a factor of $ax^n+by^m-1$, then $a \in \mathbb{F}_{p^r}$ and $b \in \mathbb{F}_{p^{2r}}$ with $b^{p^r}=-b$ (in particular, $2r \mid h$). The converse is straightforward. 
\end{proof}

\begin{proposition}\label{Prop2.5'}
Let $p>5$ and suppose that $p \mid (n-1)$ and $p \mid (m-2)$. Then the curve $\f$ defined over $\F$ is $\F$-Frobenius classical w.r.t. $\Sigma_2$.
\end{proposition}
\begin{proof}
As in the previous cases, Propositions \ref{Prop2.5} and \ref{Prop2.6} (b) give that $\f$ is $\F$-Frobenius nonclassical if and only if
\begin{equation}\label{E2.C}
   a x^{n-1+q} + b y^{m-2+2q} - 1 = 0.
\end{equation}

Using an argument analogous to that in Proposition \ref{Prop2.1'}, we conclude that $m = 2n$, and since we are assuming that $n \geq m$, this leads to a contradiction.

\end{proof}

\begin{proposition}\label{Prop2.6'}
Let $p>5$ and suppose that $p \mid (n-1)$ and $p \mid (m+1)$. Then the curve $\f$ defined over $\F$ is $\F$-Frobenius classical w.r.t. $\Sigma_2$.
\end{proposition}
\begin{proof}
    By Propositions \ref{Prop2.5} and \ref{Prop2.6} (b), $\f$ is $\F$-Frobenius nonclassical w.r.t. $\Sigma_2$ if and only if
\begin{equation}\label{E2.A}
   a x^{n-1+q} y^q + b y^{m+1} - y^q = 0.
\end{equation}
The proof is then analogous to the one of Proposition \ref{Prop2.10'}.
\end{proof}

Propositions \ref{Prop2.1'} to \ref{Prop2.6'} immediately imply the main result of this section.

\begin{theorem}\label{Teo2.4}
Suppose that $p>5$ and $q=p^h$. If the curve $\f$ defined over $\F$ is classical w.r.t. $\Sigma_1$, then $\f$ is $\F$-Frobenius nonclassical w.r.t. $\Sigma_2$ if and only if one of the following conditions holds:
\begin{itemize}
    \item [(a)] $p \mid (n-2)$ and $m=n=\frac{2(q-1)}{p^r-1}$ with $r<h$ such that $r \mid h$ and $a,b \in \mathbb{F}_{p^r}$;
    \item [(b)] $p \mid (2n-1)$ and $m=n=\frac{q-1}{2(p^r-1)}$ with $r<h$ such that $r \mid h$ and $a^2,b^2 \in \mathbb{F}_{p^r}$;
    \item [(c)] $p \mid (n+1)$ and $q=n+1=m+1$ and $a+b=1$;
    \item [(d)] $p \mid (n-1)$ and $n=2m=\frac{q-1}{p^r-1}$ with $r<h$ such that $r \mid h$ and either $a,b \in \mathbb{F}_{p^r}$ or $2r \mid h$, $a \in \mathbb{F}_{p^r}$ and $b \in \mathbb{F}_{p^{2r}}$ with $b^{p^r}=-b$;
    \item [(e)] $p \mid (n-2)$ and $n=2m=\frac{2(q-1)}{p^r-1}$ with $r<h$ such that $r \mid h$ and $a,b \in \mathbb{F}_{p^r}$.
\end{itemize}
\end{theorem}

\subsection{The number of rational points}
As pointed out in the Introduction,  on the one hand, the classification of Frobenius nonclassical curves allows us to establish a nice upper bound for the number of rational points in the classical cases. On the other hand, the Frobenius nonclassical cases tend to have many rational points. In this section, the following result will be obtained.

\begin{theorem}\label{pr}
Let $\f: aX^n + bY^mZ^{n-m} = Z^n$ be a projective curve defined over $\F$, where $a,b \in \F^*$ and $m,n$ are positive integers.  
Assume that $p > 5$ and that $\f$ is classical w.r.t. $\Sigma_1$. Then one of the following cases occurs:
    \begin{itemize}

        \item[(i)] If $n = m = \frac{q - 1}{2(p^r - 1)}$ with $r < h$, $r \mid h$, and $a^2, b^2 \in \mathbb{F}_{p^r}$, then 
        $$N_q(\f) = n^2(p^r - 2) + 3n;$$
        
        \item[(ii)] If $n=m=\frac{2(q-1)}{p^r-1}$ with $r<h$ such that $r \mid h$, and $a,b \in \mathbb{F}_{p^r}$, then 
\[
   N_q(\f) = \frac{n^2}{4} \big(p^r + 1 - 2(\psi(a)+\psi(b)+\psi(-ab))\big) + n(\psi(a)+\psi(b)+\psi(-ab))
\]    
 where $\psi(\alpha)=1$ if $\alpha$ is a square in $\mathbb{F}_{p^r}$, and $\psi(\alpha)=0$ otherwise

        \item[(iii)] If $q = n + 1 = m + 1$ and $a + b = 1$, then 
        $$N_q(\f) = (q - 1)^2;$$   
        
        \item[(iv)] If $n=2m=\frac{q-1}{p^r-1}$ with $r<h$ such that $r \mid h$ and either $a,b \in \mathbb{F}_{p^r}$ or $2r \mid h$, $a \in \mathbb{F}_{p^r}$ and $b \in \mathbb{F}_{p^{2r}}$ with $b^{p^r}=-b$, then 
        $$N_q(\f) = \frac{n^2}{2}(p^r - 2) + 2n;$$ 
        
        \item[(v)] If $n = 2m = \frac{2(q - 1)}{p^r - 1}$ with $r < h$, $r \mid h$, and $a,b \in \mathbb{F}_{p^r}$, then 
        $$
        N_q(\f) =
        \begin{cases}
            m^2(p^r - 3) + 4m, & \text{if $a$ is a square in $\mathbb{F}_{p^r}$}\\
            m^2(p^r - 1) + 2m, & \text{if $a$ is not a square in $\mathbb{F}_{p^r}$}
        \end{cases};
        $$
        
        \item[(vi)] In all remaining cases, an upper bound for $N_q(\f)$ is given by:
        \begin{eqnarray}\label{cotaSV}
    N_q(\f) &\leq &\frac{10(mn - m - n - \gcd(m, n)) + (q + 5)2n}{5}
    - \frac{\alpha(4m - 11) + (n - \alpha)(2m - 6)  }{5} \nonumber \\
    & - & \frac{\beta(4n - 11)+(m - \beta)(2n - 6)}{5}. 
\end{eqnarray}
              where $\alpha$ and $\beta$ denote the number of roots in $\F$ of the polynomials $T^n - a^{-1}$ and $T^m - b^{-1}$, respectively.   
    \end{itemize}
\end{theorem}

The cases (i) and (ii) of Theorem \ref{pr} can be found in \cite[Section 2]{GV1}. In what follows, we proceed to prove the remaining items.

\begin{proposition}
    
\label{Teo2.8} 
Let $q = n+1 = m+1$ and $a+b=1$. Then
\[
   N_q(\f) = (q-1)^2.
\]   
\end{proposition}

\begin{proof} 
Consider the Fermat curve $\f: a x^{q-1} + (1-a) y^{q-1} = 1$. Note that this curve has no points with a zero coordinate. Thus, if $\alpha,\beta \in \F$ are both nonzero, then
\[
   a \alpha^{q-1} + (1-a) \beta^{q-1} = a + (1-a) = 1.
\]
Since there are $q-1$ elements in $\mathbb{F}_q^*$, the result follows.
\end{proof}

\begin{proposition}
    
\label{Teo2.9}
Let $n=2m=\frac{q-1}{p^r-1}$ with $r<h$ such that $r \mid h$ and either $a,b \in \mathbb{F}_{p^r}$ or $2r \mid h$, $a \in \mathbb{F}_{p^r}$ and $b \in \mathbb{F}_{p^{2r}}$ with $b^{p^r}=-b$.  
Then
\[
N_q(\f) = \frac{n^2}{2}(p^r - 2) + 2n.
\]
\end{proposition}

\begin{proof}
First assume that $a,b \in \mathbb{F}_{p^r}$. Consider the projective curve $\f: aX^{n} + bY^{\frac{n}{2}}Z^{\frac{n}{2}} = Z^{n}$. Its only singular point is $P = (0:1:0)$, which is the only point of $\f$ at the infinity. Since the norm map $\F \rightarrow \mathbb{F}_{p^r}$ is surjective, there exists $\alpha \in \F$ such that $a = \alpha^n$.  
Hence, the number of $\F$-rational points on $\f$ coincides with the number of $\F$-rational points on the projective curve 
\[
\mathcal{D}: X^n + bY^{\frac{n}{2}}Z^{\frac{n}{2}} = Z^n,
\]
also defined over $\F$.  
We start by counting the affine rational points $(u,v)$ of $\mathcal{D}$, which are all smooth.  
When $uv = 0$, the number of rational points equals
\[
n + \frac{n}{2} = \frac{3n}{2}.
\]
Now let $(u,v)$ be an $\F$-rational affine point of $\f$ such that $uv \neq 0$. Then
\[
u^n + b v^{\frac{n}{2}} = 1 \iff b v^{\frac{n}{2}} = 1 - u^n,
\]
and therefore $v^{\frac{n}{2}} \in \mathbb{F}_{p^r}$.  
In this case, there are
\[
n \cdot \frac{n}{2}(p^r - 2) = \frac{n^2}{2}(p^r - 2)
\]
rational points.  
Consequently, the number of rational non-singular points is
\[
\frac{n^2}{2}(p^r - 2) + \frac{3n}{2}.
\]

We now determine the number of $\F$-rational branches of $\mathcal{D}$ centered at $P$. Note that these branches correspond to the poles of the functions $x$ and $y$. From \cite[Proposition 6.3.1]{Sti}, these functions have precisely $n/2$ distinct poles. Thus there are precisely $n/2$ distinct branches centered at $P$.
De-homogenizing $\mathcal{D}$ with respect to the variable $Y$, we obtain
$
X^n + b Z^{\frac{n}{2}} - Z^n = 0.
$
The multiplicity of $P$ is $n/2$, and then all such branches are linear. From \cite[Theorem 8.10]{HKT}, we conclude that these branches are all defined over $\F$. Hence, the result follows.

Now suppose that $2r \mid h$, $a \in \mathbb{F}_{p^r}$ and $b \in \mathbb{F}_{p^{2r}}$ with $b^{p^r}=-b$. We claim that the polynomial $T^{n/2}-b$ always has roots in $\F$. Indeed, the condition $b^{p^r}=-b$ means that the order of $b$ in $\F^{*}$ divides $2(p^r-1)$. If $\delta$ is a generator of $\F^{*}$, we know that the subgroup of $\F^{*}$ containing all elements whose order divides $2(p^r-1)$ is $\langle \delta^{n/2}\rangle$. Hence $b \in \langle\delta^{n/2}\rangle$, proving the claim. In particular, if $(u,v)$ is an $\F$-rational affine point of $\f$ such that $uv \neq 0$, then
\[
u^n + b v^{\frac{n}{2}} = 1 \iff  v^{\frac{n}{2}} = \frac{1 - u^n}{b},
\]
where $\left(\frac{1 - u^n}{b}\right)^{p^r}=-\frac{1 - u^n}{b}$. The proof is then analogous to the one of the previous case.
\end{proof}

\begin{proposition}\label{Teo2.10}
Let $n = 2m = \frac{2(q-1)}{p^r - 1}$ with $r < h$, such that $r \mid h$, and let $a,b \in \mathbb{F}_{p^r}$. Then
\[
N_q(\f) =
\begin{cases}
    m^2(p^r - 3) + 4m, & \text{if $a$ is a square in $\mathbb{F}_{p^r}$}\\[3pt]
    m^2(p^r - 1) + 2m, & \text{if $a$ is not a square in $\mathbb{F}_{p^r}$}
\end{cases}.
\]
\end{proposition}

\begin{proof}
Consider the projective curve $\f: aX^{2m} + bY^mZ^m = Z^{2m}$. Its only singular point is $P = (0:1:0)$, which is the only point of $\f$ at the infinity.

Assume first that $a$ is a square in $\mathbb{F}_{p^r}$.  
Since the norm map is surjective, there exist $\alpha, \beta \in \F$ such that $a = \alpha^{2m}$ and $b = \beta^m$.  
Hence, the number of $\F$-rational points on $\f$ equals the number of $\F$-rational points on the projective curve
\[
\mathcal{D}: X^{2m} + Y^mZ^m = Z^{2m},
\]
also defined over $\F$.

If $a$ is not a square in $\mathbb{F}_{p^r}$, let $b = \beta^m$ for some $\beta \in \F$; in this case, the number of $\F$-rational points on $\f$ equals the number of $\F$-rational points on
\[
\mathcal{D}: aX^{2m} + Y^mZ^m = Z^{2m},
\]
also defined over $\F$.
We proceed using the same reasoning as in the previous Proposition.

Consider first the non-singular points of $\mathcal{D}$ in both cases.  
The number of $\F$-rational points of $\mathcal{D}$ on the conic $XY = 0$  is $m + 2m = 3m$ if $a$ is a square in $\mathbb{F}_{p^r}$, and $m$ otherwise.

De-homogenizing with respect to $Z=1$, let $(u,v)$ with $uv \neq 0$ be an $\F$-rational point of $\f$.  
Then
\[
u^{2m} + v^m = 1, \quad \text{if $a$ is a square in $\mathbb{F}_{p^r}$,}
\]
and
\[
a u^{2m} + v^m = 1, \quad \text{if $a$ is not a square in $\mathbb{F}_{p^r}$.}
\]
In these cases, there are $m \cdot m(p^r - 3) = m^2(p^r - 3)$ rational points when $a$ is a square  in $\mathbb{F}_{p^r}$, and $(p^r - 1)m^2$ when $a$ is not a square.  
Therefore, the total number of affine rational points is
\[
N_1(\f) =
\begin{cases}
    m^2(p^r - 3) + 3m, & \text{if $a$ is a square  in $\mathbb{F}_{p^r}$}\\[3pt]
    m^2(p^r - 1) + m, & \text{if $a$ is not a square  in $\mathbb{F}_{p^r}$}
\end{cases}.
\]
Now, arguing exactly as in the proof of Proposition \ref{Teo2.9}, we conclude that in both cases there are $m$ $\F$-rational branches centered at $P$, giving the desired equalities.
\end{proof}

Finally, assume that $\f$ is $\F$-Frobenius classical w.r.t. $\Sigma_2$. Using the same notation as given in the Introduction and Proposition \ref{Prop2.2}, we have the following expressions:

\begin{equation*}
    A(P_{\xi}) =
    \begin{cases}
        4n - 11, &\text{if $P_{\xi} \in \f (\F)$}\\
        2n - 6, &\text{otherwise}
    \end{cases};
\end{equation*}

\begin{equation*}
    A(P_{\rho}) =
    \begin{cases}
        4m - 11, &\text{if $P_{\rho} \in \f (\F)$}\\
        2m - 6, &\text{otherwise}
    \end{cases}.
\end{equation*}

Let $\alpha$ and $\beta$ be the number of roots in $\F$ of the polynomials $T^n - a^{-1}$ and $T^m - b^{-1}$, respectively. According to \cite{AP}, the genus of the curve is given by 
$$
g = \frac{mn - m - n - \gcd(m, n) + 2}{2}.
$$
Then, \eqref{Eq4} in this case reads as \eqref{cotaSV}. The proof of Theorem \ref{pr} follows by gathering all the results of the section.

\section{Curve $\C: ax^n y^m + bx^n + cy^m = 1$}

We now consider the curve $\f: ax^n y^m + bx^n + cy^m = 1$ defined over $\F$, where $a,b,c \in \F$, $c \ne -\frac{a}{b}$ and $m,n$ are positive integers, both not divisible by $p$.  If any of the coefficients $a,b,c$ is zero, then one can easily check that $\f$ is $\F$-birationally equivalent to a curve of the type $dx^n+ey^m=1$. Hence, from now on we always assume $a,b,c$ to be non-zero. Further, we also assume without loss of generality $n \geq m$.  
 Since the classicality of the curve is a geometric property, for Proposition \ref{Prop3.1} and Subsection \ref{ss4.2}, we shall assume $b = c = 1$ (in particular, $a \ne -1$).
\subsection{Classicality and Frobenius classicality of $\f$ w.r.t. $\Sigma_1$}

\begin{proposition}\label{Prop3.1}
Assume $p>2$. For all positive integers $m,n$ with $p \nmid mn$, the curve $\f$ defined over $\F$ is classical w.r.t. $\Sigma_1$.
\end{proposition}

\begin{proof}
Since $p \nmid n$, $x$ is a separable variable. Note that the nonclassicality of $\f$ is equivalent to $D_x^{(2)}y = 0$. Using the equation $ax^n y^m + x^n + y^m = 1$, we obtain
\begin{equation}\label{Eq3.1}
    D_x^{(1)}y=\frac{ny(1-y^m)}{mx(x^n-1)}.
\end{equation}
Applying $D_x^{(1)}$ again gives us
\begin{equation} \label{Eq3.2}
D_x^{(1)}y \big(n - n(m+1)y^m + m - m(n+1)x^n\big) + 2 m x D_x^{(2)}y \big(1 - x^n\big) = 0.    
\end{equation}
Thus $\f$ is nonclassical if and only if
\begin{align*}    
D_x^{(1)}y \big(n - n(m+1)y^m + m - m(n+1)x^n\big) = 0.
\end{align*}
Since  $D_x^{(1)}y$ is non-zero, we conclude that
\begin{equation}\label{Eq3.3}
(n+m) - n(m+1)y^m - m(n+1)x^n = 0.
\end{equation}
However, $\deg \f=m+n$, and then equation \eqref{Eq3.3} is possible if and only if $p \mid m+1$, $p \mid n+1$ and $p \mid m+n$. In turn, since $p \nmid mn$, the last assertion is equivalent to $p=2$, which is a contradiction.
\end{proof}

From Propositions \ref{Prop1.4} and \ref{Prop3.1}, we immediately obtain the following.

\begin{proposition}
    Assume $p>2$. For all positive integers $m,n$ with $p \nmid mn$, the curve $\f$ defined over $\F$ is $\F$-Frobenius classical w.r.t. $\Sigma_1$.
\end{proposition}

\subsection{Classicality of $\f$ w.r.t. $\Sigma_2$}\label{ss4.2}

In order to investigate the nonclassicality of $\f$ w.r.t. $\Sigma_2$, we will use the next result, which is an adaptation of \cite[Proposizione 3.3]{AK}. The proof is analogous and will be omitted. 

\begin{proposition} \label{Prop3.10}
Let $\f: aX^n Y^m + bX^n Z^m + cY^m Z^n - Z^{m+n} = 0$ be a projective curve defined over $\F$, where $a,b,c \in \F^{*}$, $c \ne -\frac{a}{b}$, and $m,n$ are positive integers with $p \nmid mn$.  
Assuming $p > 5$ and that $\f$ is classical w.r.t. $\Sigma_1$, we have:
\begin{itemize}
    \item[(a)] $\f$ is geometrically irreducible;
    \item[(b)] The genus of $\f$ is $(n-1)(m-1)$;
    \item[(c)] The only singular points of $\f$ are $P=(1:0:0)$ and $Q=(0:1:0)$, with $m_P(\f) = m$ and $m_Q(\f) = n$, respectively. Moreover, the tangent lines to $\f$ at $P$ and $Q$ are given by the affine equations $y = \alpha$ and $x = \beta$, respectively, where $\alpha^m = -b a^{-1}$ and $\beta^n = -c a^{-1}$. These tangent lines intersect $\f$ at the corresponding points with multiplicity $m+n$;
    \item[(d)] The intersection multiplicity of a branch centered at $P$ and $Q$ with its tangent line is $m+1$ and $n+1$, respectively;
    \item[(e)] The points $P_{\xi} = (0:\xi:1)$ and $P_{\rho} = (\rho:0:1)$, with $\xi^m = c^{-1}$ and $\rho^n = b^{-1}$, are inflection points of $\f$. Furthermore, the tangent lines to $\f$ at $P_{\xi}$ and $P_{\rho}$ are given by $y = \xi$ and $x = \rho$, respectively, and these lines intersect $\f$ with multiplicities $n$ and $m$, respectively.
\end{itemize}    
\end{proposition}

\begin{proposition}\label{Prop3.2}
If $p > 5$ and the curve $\f$ defined over $\F$ is nonclassical w.r.t. $\Sigma_2$, then 
\[
p \mid (n+1)(n-1) \quad \text{and} \quad p \mid (m+1)(m-1).  
\]    
\end{proposition}

\begin{proof}
Taking into account the inflection points, the $(\Sigma_1, P_{\xi})$-order sequence is given by $(0,1,n)$, with $n \geq 3$, and consequently, the $(\Sigma_2, P_{\xi})$-order sequence is $(0,1,2,n,n+1,2n)$.  
Similarly, using $P_{\rho}$, we obtain that the $(\Sigma_2, P_{\rho})$-order sequence is $(0,1,2,m,m+1,2m)$.
By Proposition \ref{Prop1.1}, it follows that if $\f$ is nonclassical w.r.t. $\Sigma_2$, then
\begin{equation} \label{Ep1}
    p \mid (2n-1)(n-2)(n+1)(n-1) \quad \text{and} \quad p \mid (2m-1)(m-2)(m+1)(m-1). 
\end{equation}

Now, let  
$\gamma$ be a branch of $\f$ centered at $P$.  
By Proposition \ref{Prop3.10}, the $(\Sigma_1, \gamma)$-order sequence is $(0,1,m+1)$, and consequently, the $(\Sigma_2, P)$-order sequence is $(0,1,2,m+1,m+2,2m+2)$.  
Analogously, if $\zeta$ is a branch of $\f$ centered at $Q$, the $(\Sigma_2, \zeta)$-order sequence is $(0,1,2,n+1,n+2,2n+2)$.

Again, using Proposition \ref{Prop1.1}, it follows that if $\f$ is nonclassical w.r.t. $\Sigma_2$, then
\begin{equation} \label{Ep2}
    p \mid (n+1)(n-1)(n+2)(2n+1) \quad \text{and} \quad p \mid (m+1)(m-1)(m+2)(2m+1).
\end{equation}

Combining \eqref{Ep1} and \eqref{Ep2}, we conclude that $p \mid (n+1)(n-1)$ and $p \mid (m+1)(m-1)$.
\end{proof}

\begin{lemma}\label{Le3.1} 
Assuming $p > 2$ and $a \neq -1$, the following projective curves are irreducible over $\overline{\mathbb{F}}_q$:
\begin{itemize}
    \item[(a)] $\f_1: aXY + XZ + YZ - Z^2 = 0$; 
    \item[(b)] $\f_2: aXZ + XY + Z^2 - YZ = 0$; 
    \item[(c)] $\f_3: aYZ + XY + Z^2 - XZ = 0$; 
    \item[(d)] $\f_4: aZ^2 + YZ + XZ - XY = 0$. 
\end{itemize}
\end{lemma}

\begin{proof}
For all cases, it suffices to observe that the conics are non-singular, and therefore irreducible.
\end{proof}

\begin{proposition} \label{Prop3.3} 
Assume $p>5$, $b=c=1$ and $a \neq -1$. The curve $\f$ is nonclassical w.r.t. $\Sigma_2$ in the following cases:
\begin{itemize}
    \item[(a)] $p \mid (n-1)$ and $p \mid (m-1);$
    \item[(b)] $p \mid (n-1)$ and $p \mid (m+1);$
    \item[(c)] $p \mid (m-1)$ and $p \mid (n+1);$
    \item[(d)] $p \mid (n+1)$ and $p \mid (m+1).$
\end{itemize}         
\end{proposition}

\begin{proof} 
The proof of each item will follow from Lemma~\ref{Le2.1a}(a), in the same fashion as Proposition \ref{Prop2.4}. Following the notation of such Lemma, in each case we provide a polynomial $G(X,Y)\in \overline{\mathbb{F}}_q[x,y][X,Y]$ and for a general point $P=(u:v:1)$, the corresponding absolutely irreducible conic $\mathcal{G}_P$. In what follows, for integers $r,s,w,l$, we set $\alpha=(u^w)^{p^s}$ and $\beta=(v^l)^{p^s}$.
\begin{itemize}
    \item[(a)] Let $r,s,k,l$ be integers such that $n = p^r k + 1$ and $m = p^s l + 1$, with $p \nmid k$ and $p \nmid l$. WLOG , assume $r \geq s$, i.e., $r = s+d$ for some integer $d$, so that $n = (p^d k) p^s + 1 = w p^s + 1$. Hence:
    \begin{equation}\label{E3.1}
        ax^n y^m + x^n + y^m - 1 = 0 
        \iff a(x^w)^{p^s} x (y^l)^{p^s} y + (x^w)^{p^s} x + (y^l)^{p^s} y - 1 = 0,
    \end{equation}
    and then we have the associated irreducible conic
    $\mathcal{G}_P: a\alpha\beta XY + \alpha XZ + \beta YZ - Z^2 = 0$.

    \item[(b)] Let $r,s,k,l$ be integers such that $n = p^r k + 1$ and $m = p^s l - 1$, with $p \nmid k$ and $p \nmid l$. WLOG, assume $r \geq s$, i.e., $r = s+d$, so that $n = (p^d k) p^s + 1 = w p^s + 1$. Then
    \begin{equation}\label{E3.2}
        ax^n y^m + x^n + y^m - 1 = 0
        \iff a(x^w)^{p^s} x (y^l)^{p^s} y^{-1} + (x^w)^{p^s} x + (y^l)^{p^s} y^{-1} - 1 = 0,
    \end{equation}
     and the associated irreducible conic is
    $\mathcal{G}_P: a\alpha\beta XZ + \alpha XY + \beta Z^2 - YZ = 0$.
 
    \item[(c)] By reasoning analogously to (b) and applying the projectivity $(X:Y:Z) \mapsto (Y:X:Z)$, we obtain the associated irreducible conic
    $\mathcal{G}_P: aYZ + XY + Z^2 - XZ = 0$.

    \item[(d)] Let $r,s,k,l$ be integers such that $n = p^r k - 1$ and $m = p^s l - 1$, with $p \nmid k$ and $p \nmid l$. WLOG, assume $r \geq s$, i.e., $r = s+d$, so that $n = (p^d k) p^s - 1 = w p^s - 1$. Then
    \begin{equation}\label{E3.3}
        ax^n y^m + x^n + y^m - 1 = 0
        \iff a(x^w)^{p^s} x^{-1} (y^l)^{p^s} y^{-1} + (x^w)^{p^s} x^{-1} + (y^l)^{p^s} y^{-1} - 1 = 0,
    \end{equation}
    and the associated irreducible conic is
    $\mathcal{G}_P: a\alpha\beta Z^2 + \alpha YZ + \beta XZ - XY = 0$.
   
  \end{itemize}
\end{proof}

\begin{theorem} \label{Teo3.2} 
Assume $p>5$, $b=c=1$ and $a \neq -1$. The curve $\f$ is nonclassical w.r.t. $\Sigma_2$ if and only if one of the following conditions holds:
\begin{itemize}
    \item[(a)] $p \mid (n-1)$ and $p \mid (m-1);$
    \item[(b)] $p \mid (n-1)$ and $p \mid (m+1);$
    \item[(c)] $p \mid (m-1)$ and $p \mid (n+1);$
    \item[(d)] $p \mid (n+1)$ and $p \mid (m+1).$
\end{itemize}          
\end{theorem}

\begin{proof} 
The proof follows directly from Propositions \ref{Prop3.2} and \ref{Prop3.3}.         
\end{proof}

\subsection{$\F$-Frobenius classicality of $\f$ with respect to $\Sigma_2$}

In this section, $\mathcal{H}_P$ will denote the osculating conic to $\f$ at the point $P$.

\begin{proposition} \label{Prop3.5}
Assume that $p>5$ and that $\f$ nonclassical w.r.t. $\Sigma_2$. Then, for any point $P=(u:v:1) \in \f$ with $uv \neq 0$, the osculating conic $\mathcal{H}_P$ to $\f$ at $P$ is the irreducible projective curve defined by $H_P(X,Y,Z)=0$ where

\[
H_P(X,Y,Z) =
\begin{cases}
au^{n-1}v^{m-1}XY + bu^{n-1}XZ + cv^{m-1}YZ - Z^2 = 0, & \text{if } p\mid (n-1)\ \text{and}\ p\mid (m-1);\\[6pt]
au^{n-1}v^{m+1}XZ + bu^{n-1}XY + cv^{m+1}Z^2 - YZ = 0, & \text{if } p\mid (n-1)\ \text{and}\ p\mid (m+1);\\[6pt]
au^{n+1}v^{m-1}YZ + bu^{n+1}Z^2 + cv^{m-1}XY - XZ = 0, & \text{if } p\mid (n+1)\ \text{and}\ p\mid (m-1);\\[6pt]
au^{n+1}v^{m+1}Z^2 + bu^{n+1}YZ + cv^{m+1}XZ - XY = 0, & \text{if } p\mid (n+1)\ \text{and}\ p\mid (m+1).
\end{cases}
\]
\end{proposition}

\begin{proof}
The proof follows from Proposition~\ref{Prop3.3} and Lemma \ref{Le2.1}.
\end{proof}

\begin{proposition} \label{Prop3.6}
Let $p>5$ and suppose that $p\mid (n-1)$ and $p\mid (m-1)$. Then, the curve $\f$ defined over $\F$ is $\F$-Frobenius nonclassical w.r.t. $\Sigma_2$ if and only if 
\[
m = n = \frac{q-1}{p^r - 1},
\]
for some $r < h$ such that $r \mid h$ and $a,b,c \in \mathbb{F}_{p^r}$.
\end{proposition}

\begin{proof}
From Propositions~\ref{Prop3.5} and~\ref{Prop2.6}(b), it follows that if $\f$ is $\F$-Frobenius nonclassical w.r.t. $\Sigma_2$ if and only if
\begin{equation}\label{E3.11}
    a x^{n-1+q} y^{m-1+q} + b x^{n-1+q} + c y^{m-1+q} - 1 = 0.
\end{equation}
Let $\alpha, \beta$ be positive integers, both co-prime to $p$, such that $n-1+q=p^r \alpha$ and $m-1+q=p^s \beta$, where $r,s$ are positive integers as well. Assume that $r>s$. Then \eqref{E3.11} implies
$$
a^{1/p^s} x^{p^{r-s}\alpha} y^{\beta} + b^{1/p^s} x^{p^{r-s}\alpha} + c^{1/p^s} y^{\beta} - 1 = 0,
$$
which implies that $y$ is not a separating element, a contradiction. Analogously, $r<s$ implies that $x$ is not separating. Thus $r=s$.
In particular, \eqref{E3.11} gives that
\begin{equation}\label{nova1}
a^{1/p^r} x^{\alpha} y^{\beta} + b^{1/p^r} x^{\alpha} + c^{1/p^r} y^{\beta} - 1 
\end{equation}
vanishes in $\F(\f)$, that is, $ax^n y^m + bx^n + cy^m - 1$ divides \eqref{nova1}. The function \eqref{nova1}, seen as a polynomial in $\F[x,y]$, is absolutely irreducible. Indeed, it has the same shape as the defining polynomial of $\f$, and then it is irreducible since $c^{1/p^r} \neq -(a/b)^{1/p^r}$. Hence we conclude that \eqref{nova1} implies $a,b,c \in \mathbb{F}_{p^r}$, $\alpha=n$ and $\beta=m$. Note that 
$$
\alpha=n \Longrightarrow n-1+q=p^rn \Longrightarrow  n=\frac{q-1}{p^r-1}.
$$
and $\beta=m$ implies $m=n=\frac{q-1}{p^r-1}$. The converse follows immediately by \eqref{E3.11}. 
\end{proof}

\begin{proposition} \label{Prop3.7}  
Let $p > 5$ and suppose that $p \mid (n-1)$ and $p \mid (m+1)$.  
Then the curve $\mathcal{F}$ defined over $\mathbb{F}_q$ is $\mathbb{F}_q$-Frobenius nonclassical w.r.t. $\Sigma_2$ if and only if $n=\frac{q-1}{p^r-1}$, $m=\frac{q-1}{p^r+1}$ for some $r<h$ such that $2r \mid h$, where $-b=ca^{p^r}$, $-a=cb^{p^r}$, $c^{p^r+1}=1$ and $a,b,c \in \mathbb{F}_{p^{2r}}$.  
\end{proposition}
\begin{proof}
From Propositions \ref{Prop3.5} and \ref{Prop2.6} (b), $\mathcal{F}$ is $\mathbb{F}_q$-Frobenius nonclassical w.r.t. $\Sigma_2$ if and only if
\begin{equation}\label{E3.16}
       a x^{n-1+q} y^{m+1} + b x^{n-1+q}y^q + c y^{m+1} - y^q = 0.
\end{equation}
Assume that $\f$ is $\mathbb{F}_q$-Frobenius nonclassical w.r.t. $\Sigma_2$. Then $q \geq m+1$. Indeed, if $q< m+1$, then it follows from \eqref{E3.16} that 
$$
a x^{n-1+q} y^{m+1-q} + b x^{n-1+q} + c y^{m+1-q} - 1 = 0,
$$
thus $[\F(\f): \F(x)] \leq m+1-q < m =[\F(\f): \F(x)]$, a contradiction. Hence $q\geq m+1$ and from \eqref{E3.16} we have
\begin{equation}\label{E3.16.2}
       a x^{n-1+q}  + b x^{n-1+q}y^{q-m-1} + c - y^{q-m-1} = 0.
\end{equation}
Let $\alpha,\beta$ be positive integers, both co-prime to $p$, such that $n-1+q=p^r \alpha$ and $q-m-1=p^s\beta$. The same argument used in the proof of Proposition \ref{Prop3.6} leads us to $r=s$. From \eqref{E3.16.2} we have
\begin{equation}\label{E3.16.3}
       \left(\frac{-a}{c}\right)^{1/p^r} x^{\alpha}  + \left(\frac{-b}{c}\right)^{1/p^r} x^{\alpha}y^{\beta} -1 +\left(\frac{1}{c}\right)^{1/p^r} y^{\beta} = 0.
\end{equation}
Since $c \neq -a/b$, one can see that the left side of \eqref{E3.16.3} is an absolutely irreducible polynomial in $\F[x,y]$ (same argument used in the proof of Proposition \ref{Prop3.6}). Hence \eqref{E3.16.3} implies that $\alpha=n$ and $\beta=m$, and then $n=\frac{q-1}{p^r-1}$ and $m=\frac{q-1}{p^r+1}$ (which in particular implies that $2r \mid h$). Furthermore, we conclude that 
$$
\left(\frac{-a}{c}\right)^{1/p^r}=b, \ \ \left(\frac{-b}{c}\right)^{1/p^r}=a \ \ \text{ and } \ \ c^{p^r+1}=1.
$$
Thus $(c^{p^r+1})^{p^r-1}=1$, that is, $c \in \mathbb{F}_{p^{2r}}$. From $b=-a^{p^r}c$ we have $b^{p^r}=-a^{p^{2r}}c^{p^r}=-a/c$, and then $a^{p^{2r}-1}=1$. Hence $a \in \mathbb{F}_{p^{2r}}$ and consequently $b \in \mathbb{F}_{p^{2r}}$. The converse follows straightforwardly by reversing the steps from \eqref{E3.16.3}. 
\end{proof}

\begin{proposition} \label{Prop3.8}  
Let $p > 5$ and suppose that $p \mid (n+1)$ and $p \mid (m-1)$.  
Then the curve $\mathcal{F}$ defined over $\mathbb{F}_q$ is $\mathbb{F}_q$-Frobenius classical w.r.t. $\Sigma_2$.  
\end{proposition}
\begin{proof}
Arguing analogously as in Proposition \ref{Prop3.7}, we arrive at $m=\frac{q-1}{p^r-1}$ and $n=\frac{q-1}{p^r+1}$, contradicting our assumption $n \geq m$.
\end{proof}

\begin{proposition} \label{Prop3.9}  
Let $p > 5$ and suppose that $p \mid (n+1)$ and $p \mid (m+1)$.  
Then the curve $\mathcal{F}$ defined over $\mathbb{F}_q$ is $\mathbb{F}_q$-Frobenius nonclassical w.r.t. $\Sigma_2$  
if and only if $m = n = \dfrac{q-1}{p^r + 1}$ for some $r < h$ such that $2r \mid h$, with  
$ac^{p^r} = -b$, $a^{p^r + 1} = 1$, $ab^{p^r} = -c$, and $a, b, c \in \mathbb{F}_{p^{2r}}$.  
\end{proposition}
\begin{proof}
Again, from Propositions \ref{Prop3.5} and \ref{Prop2.6} (b), $\mathcal{F}$ is $\mathbb{F}_q$-Frobenius nonclassical w.r.t. $\Sigma_2$ if and only if
\begin{equation}\label{E3.17}
       a x^{n+1} y^{m+1} + b x^{n+1}y^q + c y^{m+1}x^q - x^qy^q = 0.
\end{equation}
The same argument used in the previous propositions shows that $q \geq n+1 \geq m+1$. Suppose that $\f$ is $\F$-Frobenius nonclassical w.r.t. $\Sigma_2$. Then equation \eqref{E3.17} gives
\begin{equation}\label{E3.17.1}
       a^{-1} x^{q-n-1} y^{q-m-1}+(-ca^{-1}) x^{q-n-1} + (-ba^{-1}) y^{q-m-1}  - 1 = 0.
\end{equation}
Let $\alpha,\beta$ be positive integers, both co-prime to $p$, such that $q-n-1=p^r \alpha$ and $q-m-1=p^s\beta$. Again, the same argument used in the proof of Proposition \ref{Prop3.6} leads us to $r=s$. Thus \eqref{E3.17.1} provides
\begin{equation}\label{E3.17.2}
       (a^{-1})^{1/p^r} x^{\alpha} y^{\beta}+(-ca^{-1})^{1/p^r} x^{\alpha} + (-ba^{-1})^{1/p^r} y^{\beta}  - 1 = 0.
\end{equation}
The left side of \eqref{E3.17.2} has the same shape as the defining polynomial of $\f$. Thus from $c \neq -a/b$, we conclude that the left side of \eqref{E3.17.2} is an absolutely irreducible polynomial in $\F[x,y]$. Hence \eqref{E3.17.2} implies that $\alpha=n$, $\beta=m$, $a^{p^r}=a^{-1}$, $b^{p^r}=-c/a$ and $c^{p^r}=-b/a$. The conclusion then follows analogously to the proof of Proposition \ref{Prop3.7}. The converse follows straightforwardly by reversing the steps from \eqref{E3.17.2}. 
\end{proof}

\begin{obs}
In \cite[Sections 4 and 5]{AK} curves defined by $\f_{u,v}:uX^nY^n-(X^n+Y^n)+v=0$, where $u,v \in \F$ with $uv \neq 0$ were considered from the point of view of Frobenius nonclassicality with respect to conics. The results presented in this section generalize those obtained in \cite[Sections 4 and 5]{AK}.

 Moreover, the result established in Proposition \ref{Prop3.6} was apparently overlooked in \cite[Proposizione 5.3]{AK}. In fact, in \cite[Proposizione 5.3]{AK} the authors claim that $\f_{u,v}$ is $\F$-Frobenius classical with respect to the linear system of conics provided that $p>2$, $n<(q-1)/2$ and $p \mid n-1$. However, this does not necessarily hold, as shown by Proposition \ref{Prop3.6}.
\end{obs}

\begin{theorem} \label{Teo3.2}  
Assume $p > 5$ and $q = p^h.$  
The curve $\mathcal{F}$ defined over $\mathbb{F}_q$ is $\mathbb{F}_q$-Frobenius nonclassical w.r.t. $\Sigma_2$ if and only if one of the following conditions holds:
\begin{itemize}
    \item[(i)] $p \mid (n - 1)$ and $m = n = \dfrac{q - 1}{p^r - 1}$ with $r < h$ such that $r \mid h$, and $a, b, c \in \mathbb{F}_{p^r};$
    \item [(ii)] $p \mid (n-1)$ and $p \mid (m+1)$ and  
 $n=\frac{q-1}{p^r-1}$, $m=\frac{q-1}{p^r+1}$ with $r<h$ such that $2r \mid h$, where $-b=ca^{p^r}$, $-a=cb^{p^r}$, $c^{p^r+1}=1$ and $a,b,c \in \mathbb{F}_{p^{2r}}$;
    \item[(iii)] $p \mid (n + 1)$ and $m = n = \dfrac{q - 1}{p^r + 1}$ with $r < h$ such that $2r \mid h$,  
    $a c^{p^r} = -b$, $a^{p^r + 1} = 1$, $a b^{p^r} = -c$, and $a, b, c \in \mathbb{F}_{p^{2r}}.$
\end{itemize}  
\end{theorem}

\begin{proof}
    The proof follows directly from Propositions \ref{Prop3.6}, \ref{Prop3.7}, \ref{Prop3.8}, and \ref{Prop3.9}.
\end{proof}

\subsection{The number of rational points}

We aim now to determine $N_q(\f)$ for the cases described in Theorem~\ref{Teo3.2}. For the remaining cases, a bound for $N_q(\f)$ will be presented.

\begin{proposition}
 \label{Teo3.3}
Let $n=m=\frac{q-1}{p^r-1}$ with $r<h$, $r\mid h$, and let $a,b,c \in \mathbb{F}_{p^r}$. Then
\begin{equation*}
    N_q(\f)=n^2(p^r-3)+4n.
\end{equation*}
\end{proposition}

\begin{proof}
Since $b,c \in \mathbb{F}_{p^r}$, using the surjectivity of the norm map $\F \rightarrow \mathbb{F}_{p^r}$ we may assume $b=c=1$. Consider the projective curve
\[
\f: aX^nY^n+X^nZ^n+Y^nZ^n=Z^{2n}.
\]
Its only singularities are $P=(0:1:0)$ and $Q=(1:0:0)$.

We first determine the number of $\F$-rational branches of $\f$ centered at $P$. De-homogenizing the defining polynomial of $\f$ with respect to the variable $Y$, we obtain that
\[
\f: ax^n+x^nz^n+z^n-z^{2n}=0.
\]
Thus $P$ is an ordinary singularity of multiplicity $n$, and all the $n$ tangent lines are defined over $\F$. Hence, from [11, Theorem 8.10], there are exactly $n$ branches centered at $P$, all of them linear and rational.
By an analogous argument, we conclude that there are exactly $n$ branches centered at $Q$, all of them linear and rational.

It remains only to count the number of affine rational points on $\f$, which are all nonsingular.

If $xy=0$, the number of rational points is $n+n=2n$. Now consider a $\F$-rational point $(x,y)$ of $\f$ such that $xy\neq 0$. In this case,
\[
ax^ny^n+x^n+y^n=1 \iff x^n(ay^n+1)=1-y^n.
\]
There are $n^2(p^r-3)$ rational points satisfying this condition. Therefore, the total number of rational points corresponding to nonsingular points is
\[
n^2(p^r-3)+2n.
\]

Finally, adding the contributions from the singular points, we obtain the desired result.
\end{proof}

\begin{proposition} \label{Teo3.5}
Let $m=n=\frac{q-1}{p^r+1}$ with $r<h$ such that $2r\mid h$ where $ac^{p^r}=-b$, $a^{p^r+1}=1$, $ab^{p^r}=-c$, and $a,b,c \in \mathbb{F}_{p^{2r}}$. Then
\[
N_q(\f)=\left(\frac{q-1}{p^{2r}-1}\right)^2N^{\prime}_{p^{2r}}(\mathcal{C})+(\delta(c)+\delta(b)+\delta(-c/a)+\delta(-b/a))n,
\]
where $N^{\prime}_{p^{2r}}(\mathcal{C})$ stands for the number of affine $\mathbb{F}_{p^{2r}}$-rational points of $\mathcal{C}:ax^{p^r-1}y^{p^r-1}+bx^{p^r-1}+cy^{p^r-1}-1=0$ with $xy \neq 0$ and $\delta:\mathbb{F}_{p^{2r}}^{*} \rightarrow \{0,1\}$ is defined by $\delta(d)=1$ if $d$ is a $(p^r-1)$-th power and $\delta(d)=0$ otherwise.
\end{proposition}

\begin{proof}
The only singularities of $\f$ are $P=(0:1:0)$ and $Q=(1:0:0)$. Arguing as in the proof of Proposition \ref{Teo3.3}, we conclude that both $P$ and $Q$ are ordinary. The tangent lines to $\f$ at $P$ are $X+\xi^i\alpha Z=0$, with $i=0,\ldots,n-1$, where $\xi \in \F$ is a primitive $n$-th root of unity and $\alpha^n=-c/a$. Set $e=\frac{q-1}{p^{2r}-1}$. Then $n=e(p^r-1)$, and by the surjectivity of the norm map $\F \rightarrow \mathbb{F}_{p^{2r}}$ given by $d \mapsto d^e$ we conclude that all the tangent lines to $\f$ at $P$ are defined over $\F$ if and only if $-c/a$ is a $(p^r-1)$-th power. Moreover, either all or none of such lines are defined over $\F$. An analogous argument applies to the tangent lines to $\f$ at $Q$.

We now turn our attention to the nonsingular points of the curve $\f$, which are precisely the affine points of $\f$. The curve $\f$ has an affine $\F$-rational point $(0,y)$ if and only if $y \in \F$ satisfies $y^n=1/c$, which in turn holds if and only if $c$ is a $(p^r-1)$-th power. An analogous argument holds for affine rational points of type $(x,0)$.

Now note that $n=e(p^r-1)$. Therefore, the affine equation of the curve reads
\[
ax^{e(p^r-1)}y^{e(p^r-1)}+bx^{e(p^r-1)}+cy^{e(p^r-1)}=1.
\]
Therefore the result follows by using again the surjectivity of the norm.
\end{proof}

The proof of the following result is analogous to the one of Proposition \ref{Teo3.5}, and will be omitted. 

\begin{proposition}\label{ult}
Let  $n=\frac{q-1}{p^r-1}$, $m=\frac{q-1}{p^r+1}$ with $r<h$ such that $2r \mid h$, where $-b=ca^{p^r}$, $-a=cb^{p^r}$, $c^{p^r+1}=1$ and $a,b,c \in \mathbb{F}_{p^{2r}}$. Then
\[
N_q(\f)=\left(\frac{q-1}{p^{2r}-1}\right)^2N^{\prime}_{p^{2r}}(\mathcal{G})+(\delta(c)+\delta(-b/a))m+(\gamma(-c/a)+\gamma(b))n,
\]
where $N^{\prime}_{p^{2r}}(\mathcal{G})$ stands for the number of affine $\mathbb{F}_{p^{2r}}$-rational points of $\mathcal{G}:ax^{p^r+1}y^{p^r-1}+bx^{p^r+1}+cy^{p^r-1}-1=0$ with $xy \neq 0$, $\delta:\mathbb{F}_{p^{2r}}^{*} \rightarrow \{0,1\}$ is defined by $\delta(d)=1$ if $d$ is a $(p^r-1)$-th power and $\delta(d)=0$ otherwise and $\gamma(d)=1$ if $d \in \mathbb{F}_{p^r}$ and $\gamma(d)=0$ otherwise.
\end{proposition}

As in the previous section, apart from the cases presented Theorem \ref{Teo3.2}, we are able to obtain an upper bound for $N_q(\f)$ via the St\"ohr-Voloch approach. Using again the definitions given in the introduction of this work, together with the notation of the inflection points of $\f$ presented in Proposition \ref{Prop3.10},  we obtain the following expressions:
\begin{equation*}
        A(P_{\xi})=
        \begin{cases}
            4n-11, &\text{if $P_{\xi} \in \f (\F)$}\\
            2n-6, &\text{otherwise}
        \end{cases};
    \end{equation*}
    \begin{equation*}
        A(P_{\rho})=
        \begin{cases}
            4m-11, &\text{if $P_{\rho} \in \f (\F)$}\\
            2m-6, &\text{otherwise}
        \end{cases}.
    \end{equation*}

Let $\alpha$ and $\beta$ denote the number of roots in $\F$ of the polynomials $T^m-c^{-1}$ and $T^n-b^{-1}$, respectively. By \cite{AP}, the genus of $\f$ is $g=(n-1)(m-1)$. Hence, using \eqref{Eq4} we obtain the following result.

\begin{proposition}
  Assume that $p>5$.  Let $\f:ax^ny^m+bx^n+cy^m=1$ defined over $\F$, where $c\neq -a/b$. If $\f$ is none of the types presented in Theorem \ref{Teo3.2}, then
    \begin{eqnarray}\label{cotaSV2}
    N_q(\f) &\leq &\frac{20(mn - m - n) + 2(q + 5)(m+n)}{5}
    - \frac{\alpha(4m - 11) + (n - \alpha)(2m - 6)  }{5} \nonumber \\
    & - & \frac{\beta(4n - 11)+(m - \beta)(2n - 6)}{5}. 
\end{eqnarray}
\end{proposition}

\begin{obs}
    Denote by $\Sigma_2^{P,Q}$ the linear system of conics containing the points $P=(0:1:0)$ and $Q=(1:0:0)$. Recall that we are assuming $p>5$ and consider the curve $\f:ax^ny^m+bx^n+cy^m=1$ defined over $\F$, where $c\neq -a/b$. By Proposition \ref{Prop3.10}, all the osculating conics to $\f$ at general points belong to $\Sigma_2^{P,Q}$. Hence, from Proposition \ref{Prop2.6}(b) we conclude that $\f$ is $\F$-Frobenius nonclassical w.r.t. $\Sigma_2$ if and only if $\f$ is $\F$-Frobenius nonclassical w.r.t. $\Sigma_2^{P,Q}$. As pointed out in \cite[Section 3]{Ar}, the St\"ohr-Voloch method applied to $\Sigma_2^{P,Q}$ provides a bound potentially better than the bound obtained via $\Sigma_2$. In this direction, when $m=n$ and $b=c$ with $ab \neq 0$ and $b^2\neq -a$, in \cite[Corollary 3.3]{BM} it is obtained via $\Sigma_2^{P,Q}$ a bound for $N_q(\f)$ for the $\F$-Frobenius classical cases. Therefore, apart from the cases (i) and (iii) of Theorem \ref{Teo3.2}, the bound for $N_q(\f)$ given in \cite[Corollary 3.3]{BM} holds.
\end{obs}

\section*{Acknowledgments} 

The work of Nazar Arakelian was partially supported by the São Paulo Research Foundation (FAPESP), grant 2023/03547-2.

\text{}
 
\vspace{0,5cm}\noindent {\em Author's addresses}:

\vspace{0.2 cm} \noindent Nazar ARAKELIAN \\
Instituto de Ci\^encias Matem\'aticas e de Computa\c c\~ao
\\ Universidade de S\~ao Paulo \\ Avenida Trabalhador S\~ao-carlense, 400 \\
CEP 13566-590, S\~ao Carlos SP
(Brazil).\\
 E--mail: {\tt n.arakelian@icmc.usp.br}

\text{}
 
\vspace{0.2 cm} \noindent Leandro A. M. RODRIGUES \\
Instituto Federal de S\~ao Paulo
\\ Campus Cubat\~ao \\ Rua Maria Cristina, 50 \\
CEP 11533-160, Cubat\~ao SP
(Brazil).\\
 E--mail: {\tt leandro.albino@ifsp.edu.br}

\end{document}